\documentclass[a4paper,10pt]{iopart}

\usepackage{textcomp}
\usepackage{amsmath}
\usepackage{amsfonts}
\usepackage{amsthm}
\usepackage{amssymb}
\usepackage{color}
\usepackage[mathscr]{eucal}
\usepackage{float} 
\usepackage{enumerate}
\usepackage{graphicx}
\usepackage{cite}
\usepackage{tikz}
\usepackage{fp}

\usepackage{algorithmic}
\usepackage[ruled,section]{algorithm}

\newtheoremstyle{thm}
  {\baselineskip}
  {\baselineskip}
  {\itshape}
  {}
  {\bf}
  {.}
  {.5em}
  {}
  
\newtheoremstyle{others}
  {\baselineskip}
  {\baselineskip}
  {\upshape}
  {}
  {\bf}
  {.}
  {.5em}
  {}

\theoremstyle{thm}
\newtheorem{theorem}{Theorem}[section]
\newtheorem{proposition}[theorem]{Proposition}

\theoremstyle{others}
\newtheorem{definition}[theorem]{Definition}

\newtheorem{remark}[theorem]{Remark}

\newlength{\arraycolseptmp}

\newcommand{\Fourier}{\mathcal{F}}

\newcommand{\argmin}{\operatorname{argmin}}
\newcommand{\argsup}{\operatorname{argsup}}

\newcommand{\C}{\mathbb{C}}
\newcommand{\convbi}{\circledast}
\newcommand{\Ds}{\mathscr{D}}
\newcommand{\Es}{\mathscr{E}}
\newcommand{\Fs}{\mathscr{F}}

\newcommand{\ip}[2]{\langle #1, #2 \rangle}
\newcommand{\Ip}[2]{\Big\langle #1, #2 \Big\rangle}
\newcommand{\jacobi}{\mbox{\large$\vartheta$}_3}
\newcommand{\jinc}{\operatorname{Jinc}}
\newcommand{\N}{\mathbb{N}}
\newcommand{\norm}[1]{\|#1\|}

\newcommand{\nsr}{r_{\varepsilon / \alpha}}
\newcommand{\R}{\mathbb{R}}
\newcommand{\rd}{\;\operatorname{d}}
\newcommand{\rg}{\operatorname{rg}}
\newcommand{\set}[2]{\{#1\, | \,#2\}}
\newcommand{\Set}[2]{\big\{#1\, \big| \,#2\big\}}
\newcommand{\sinc}{\operatorname{sinc}}
\newcommand{\spa}{\operatorname{span}}
\newcommand{\sumstack}[2]{\renewcommand{\arraystretch}{0.5}
\setlength{\arraycolseptmp}{\arraycolsep}
\setlength{\arraycolsep}{0pt}
  \begin{array}{c}
    \scriptstyle #1 \\
    \scriptstyle #2
  \end{array}
  \setlength{\arraycolsep}{\arraycolseptmp}
  \renewcommand{\arraystretch}{1.0}}
\newcommand{\supp}{\operatorname{supp}}
\newcommand{\Z}{\mathbb{Z}}
\newcommand{\bigO}{\mathcal{O}}

\begin{document}

\title[Greedy Solution of Ill-Posed Problems]{Greedy Solution of Ill-Posed Problems: Error Bounds and Exact Inversion}

\author{L Denis$^{1,2}$, D A Lorenz$^3$ and D Trede$^{4,}$\footnote[7]{Author to whom correspondence shall be addressed.}}
\address{$^1$\'{E}cole Sup\'{e}rieure de Chimie Physique \'{E}lectronique de Lyon, F--69616 Lyon, France\\
  $^2$Laboratoire Hubert Curien, UMR CNRS 5516, Universit\'{e} de Lyon, F--42000 St Etienne, France\\
  $^3$TU Braunschweig, D--38092 Braunschweig, Germany\\
$^4$Zentrum f\"ur Technomathematik, University of Bremen, D--28334 Bremen, Germany}

\eads{\mailto{loic.denis@cpe.fr}, \mailto{d.lorenz@tu-braunschweig.de}, and \mailto{trede@math.uni-bremen.de}}

\begin{abstract}
The orthogonal matching pursuit~(OMP) is a greedy algorithm to solve
sparse approximation problems.  Sufficient conditions for exact
recovery are known with and without noise. In this paper we
investigate the applicability of the OMP for the solution of
ill-posed inverse problems in general and in particular for two
deconvolution examples from mass spectrometry and digital
holography respectively.

In sparse approximation problems one often has to deal with the
problem of redundancy of a dictionary, i.e.~the atoms are not linearly
independent. However, one expects them to be approximatively
orthogonal and this is quantified by the so-called incoherence.  This idea cannot be
transfered to ill-posed inverse problems since here the atoms are
typically far from orthogonal: The ill-posedness of the operator
causes that the correlation of two distinct atoms probably gets huge,
i.e.~that two atoms can look much alike.  Therefore one needs
conditions which take the structure of the problem into account and
work without the concept of coherence.  In this paper we develop
results for exact recovery of the support of noisy signals.
In the two examples in mass spectrometry and digital holography
we show that our results lead to practically relevant estimates
such that one may check a priori if the experimental setup
guarantees exact deconvolution with OMP. Especially in the example
from digital holography our analysis may be regarded as a first
step to calculate the resolution power of droplet holography.
\end{abstract}


\section{Introduction}
We consider linear inverse problems, i.e.~we are given a bounded,
injective, linear operator $K:B\to H$ mapping from a Banach space $B$
into a Hilbert space $H$. Moreover, we assume that for an unknown
$v\in \rg K$ we are given a noisy observation $v^\varepsilon$ with
$\norm{v-v^\varepsilon}\leq\varepsilon$ and try to reconstruct the
solution of
\begin{equation}
  \label{inv_prob}
  Ku = v
\end{equation}
from the knowledge of $v^\varepsilon$. We are particularly interested
in the case where the unknown solution $u$ may be expressed sparsely
in a known dictionary, i.e.~we consider that there is a family $\Es
:= \{e_i\}_{i\in \Z} \subset B$ of unit-normed vectors which span the
space in which we expect the solution and which we call
\emph{dictionary}.  With sparse we mean here that there exists a
finite decomposition of $u$ with $N$ \emph{atoms} $e_i\in\Es$,
\[
  u = \sum\limits_{i\in\Z} \alpha_i e_i
  \quad \mbox{with} \quad
  \alpha_i \in\R, \quad \norm{\alpha}_{\ell^0} =: N<\infty.
\]
In the following we denote with $I$ the support of $\alpha$, i.e.~$I =
\set{i\in\Z}{\alpha_i\neq 0}$. 
For any subset $J\subset \Z$ we denote $\Es(J):= \set{e_i}{i\in J}$. 

This setting appears in several 
signal processing problems, e.g.~in mass
spectrometry~\cite{klann2007shrink_vs_deconv} where the signal is
modeled as a sum of Dirac peaks (so-called impulse trains):
\[
  u = \sum\limits_{i\in\Z} \alpha_i \; \delta(\cdot-x_i).
\]
Other applications for instance can be found
in astronomical signal processing problems or digital holography, 
cf.~\cite{soulez2007holography}, where images arise as
superposition of characteristic functions of balls with different centers $x_i$ and
radii $r_j$, 
\[
  u = \sum\limits_{i,j\in\Z} \alpha_{i,j} \; \chi_{B_{r_j}} (\cdot - x_i).
\]
Typically $K$ does not have a continuous inverse and hence, the
solution of the operator equation~(\ref{inv_prob}) does not depend
continuously on the data.  This turns out to be a challenge for the
case where only noisy data $v^\varepsilon$ with noise level
$\norm{v-v^\varepsilon}\leq\varepsilon$ are available---as it is
always the case in praxis.  First a small perturbation $\varepsilon$
can cause an arbitrarily large error in the reconstruction $u$ of
``$Ku=v^\varepsilon$'' and second no solution $u$ exists if
$v^\varepsilon$ is not in the range of $K$.

Inverse problems formulated in Banach spaces have been of recent
interest and there are a several results which deal with solving
inverse problems formulated in Banach spaces, e.g.~results concerning
error
estimates~\cite{burger2004convarreg,resmerita2005regbanspaces,hein2008convratesbanach,hofmann2007convtikban,lorenz2008reglp,lorenz2008besovscales,Bonesky2009}
or Landweber-like iterations or minimization methods for Tikhonov
functionals, see
e.g.~\cite{schoepfer2006illposedbanach,daubechies2003iteratethresh,
  bredies2008forwardbackwardbanach,bredies2008harditer,schuster2008tikbanach,bonesky2007gencondgradnonlin,lorenz2008conv_speed_sparsity,griesse2008ssnsparsity}.

In the following, an approximate solution of ``$Ku=v^\varepsilon$'' shall
be found by deriving iteratively the correlation between the residual
and the unit-normed atoms of the dictionary
\[
  \Ds := \{d_i\}_{i\in\Z} :=  \Big\{\frac{K e_i}{\norm{K e_i}}\Big\}_{i\in\Z}.
\]
Note that since the operator $K$ is injective we get that $K e_i \neq
0$ for all $i\in\Z$ and hence the dictionary $\Ds$ is well defined.
In any step we select that unit-normed atom from the dictionary $\Ds$
which is mostly correlated with the residual, hence the name
``greedy'' method.  To stabilize the solution of
``$Ku=v^\varepsilon$'' the iteration has to be stopped early enough.

For solving the operator equation~(\ref{inv_prob}) 
with noiseless data and the case where only noisy data $v^\varepsilon$
with noise-bound $\norm{v-v^\varepsilon}\leq \varepsilon$ are available 
we use the orthogonal matching pursuit,
first proposed in the signal processing context by
Davis et al. in~\cite{mallat1994orthogonal} 
and Pati et al. in~\cite{Pati1993} as an improvement upon the matching pursuit algorithm~\cite{mallat1993matchingpursuits}:

\begin{algorithm}[H]
  \caption{Orthogonal Matching Pursuit}\label{alg_OMP}

  \begin{algorithmic}
    \REQUIRE $k:=0$ and $I^0:=\emptyset$.
      \textbf{Initialize} $r^0:=v^\varepsilon$ (resp. $r^0:=v$ for $\varepsilon=0$) and $\widehat{u}^0 := 0$.
    \WHILE{$\norm{r^k} > \varepsilon$ (resp. $\norm{r^k} \neq 0$)}
      \STATE \quad $k :=k+1,$
      \STATE \quad $i_k \in \argsup \Set{|\ip{r^{k-1}}{d_i}|}{d_i \in \Ds},$
      \STATE \quad $I^k := I^{k-1} \cup \{i_k\},$
      \STATE \textbf{Project} $u$ onto $\spa\Es(I^k)$
      \STATE \quad $\widehat{u}^k := \argmin \Set{\norm{v^\varepsilon - K \widehat{u}}^2}{\widehat{u} \in \spa\Es(I^k)},$
      \STATE \quad $r^k := v^\varepsilon - K \widehat{u}^k.$
    \ENDWHILE
  \end{algorithmic}
\end{algorithm}

Remark that in infinite dimensional Hilbert spaces the supremum
\begin{equation}\label{eq_sup_OMP}
  \sup \set{|\ip{r^{k-1}}{d_i}|}{d_i \in \Ds}
\end{equation}
does not have to be realized. Because of that OMP has a variant---called weak orthogonal
matching pursuit (WOMP)---which
does not choose the optimal atom in the sense of~(\ref{eq_sup_OMP})
but only one that is nearly optimal, i.e.~for some fixed $\omega\in(0,1]$ it chooses some $i_k\in\Z$ with
\[
  |\ip{r^{k-1}}{d_{i_k}}| \geq \omega \sup \, \set{|\ip{r^{k-1}}{d_i}|}{d_i \in \Ds}.
\]

In~\cite{Tropp2004greed} a sufficient condition for exact recovery with
algorithm~\ref{alg_OMP} is derived, and in~\cite{donoho2004stablesparse}
it is transfered to noisy signals with the concept of coherence,
which quantifies the magnitude of redundancy.
This idea cannot be transfered to ill-posed inverse problems
directly since the operator typically causes that the
correlation of two distinct atoms probably gets huge.
Therefore in~\cite{Dossal2005,Gribonval2008} the authors derive a
recovery condition which works without the concept of coherence.
For a comprehensive presentation of OMP cf.~e.g.~\cite{Mallat2009}.

The paper is organized as follows. In section~\ref{sec_ERC} we reflect
the conditions for exact recovery for OMP derived in~\cite{Tropp2004greed}
and~\cite{Dossal2005,Gribonval2008} and rewrite them in the
context of infinite-dimensional inverse problems. 
Section~\ref{sec_epsERC} contains the main theoretical results
of the paper, namely the generalization of these results to noisy
signals.  In
section~\ref{sec_massspectro} we apply
the deduced recovery conditions in the presence of noise
to an
example from mass spectrometry. Here, the data are given as sums of
Dirac peaks convolved with a Gaussian kernel.  To the end of this
section we utilize the deduced condition for simulated data of an
isotope pattern.  Another example from digital holography is concerned
in section~\ref{sec_holography}.  The data are given as sums of
characteristic functions convolved with a Fresnel function. This turns
out to be a challenge because the convolution kernel oscillates.
Similar to section~\ref{sec_massspectro} we apply the theoretical condition
to simulated data, namely to digital holograms of particles.
The two examples from mass spectrometry and digital holography illustrate
that our conditions for exact recovery lead to practically relevant
estimates such that one may check a priori if the experimental
setup guarantees exact deconvolution with OMP.
Especially in the example from digital holography our analysis may be
regarded as a first step to calculate the resolution power of
droplet holography.

\section{Exact Recovery Conditions}\label{sec_ERC}
In~\cite{Tropp2004greed}, Tropp gives a sufficient and necessary condition 
for exact recovery with OMP.
Next, we list this result in the language of infinite-dimensional inverse problems.

Define the linear continuous synthesis operator for the 
dictionary $\Ds = \{d_i\} = \{K e_i/\norm{K e_i}\}$ via
\[
  \begin{array}{rrl}
    D: & \ell^1             & \to H, \\
       & (\beta_i)_{i\in\Z} & \mapsto \sum\limits_{i\in\Z} \beta_i d_i = \sum\limits_{i\in\Z} \beta_i \tfrac{K e_i}{\norm{K e_i}}.
  \end{array}
\]
Since $D$ is linear and bounded, the Banach space adjoint operator
\[
  D^\ast: H  \to (\ell^1)^\ast = \ell^\infty
\]
exists and arises as
\[
  D^\ast v  = (\ip{v}{d_i})_{i\in\Z}
  = \big(\ip{v}{\tfrac{K e_i}{\norm{K e_i}}} \big)_{i\in\Z}.
\]
Note that the use of $\ell^1$ and its dual $\ell^\infty$ arises
naturally in this context.  Furthermore, for $J\subset \Z$ we denote
with $P_J:\ell^1\to\ell^1$ the projection onto $J$ and with
$A^\dagger$ the pseudoinverse operator of $A$.  With this notation we
state the following theorem.  

\begin{theorem}[Tropp~\cite{Tropp2004greed}]\label{theorem_ERC}
  Let $\alpha \in \ell^0$ with $\supp \alpha = I$,
  $u = \sum_{i\in\Z} \alpha_i e_i$ be the source and $v = K u$ the measured signal.
  If the operator $K:B \to H$
  and the dictionary $\Es = \{e_i\}_{i\in\Z}$ fulfill
  the \emph{Exact Recovery Condition (ERC)}
  \begin{equation}\label{eq_ERC}
    \sup\limits_{d\in \Ds(I^\complement)} \norm{(D P_I)^\dagger d}_{\ell^1}
    <1,
  \end{equation}
  then OMP with its parameter $\varepsilon$ set to 0 recovers $\alpha$ exactly.
\end{theorem}

Theorem~\ref{theorem_ERC} gives a sufficient condition for exact recovery with OMP.
In~\cite{Tropp2004greed} Tropp shows that condition~(\ref{eq_ERC})
is even necessary, in the sense that if
\[
  \sup\limits_{d\in \Ds(I^\complement)} \norm{(D P_I)^\dagger d}_{\ell^1}
  \geq 1,
\]
then there exists a signal with support $I$ for
which OMP does not recover $\alpha$ with $v = K u = K \sum \alpha_i e_i$.

The ERC~(\ref{eq_ERC}) is hard to evaluate.  Therefore Dossal and
Mallat~\cite{Dossal2005} and Gribonval and
Nielsen~\cite{Gribonval2008} derive a weaker sufficient but not
necessary recovery condition that depends on inner products of the
dictionary atoms of $\Ds(I)$ and $\Ds(I^\complement)$ only.  
\begin{proposition}[Dossal and Mallat~\cite{Dossal2005}, Gribonval and Nielsen~\cite{Gribonval2008}]\label{theorem_suf}
  Let $\alpha \in \ell^0$ with $\supp \alpha = I$,
  $u = \sum_{i\in\Z} \alpha_i e_i$ be the source and $v = K u$ the measured signal.
  If the operator $K:B \to H$
  and the dictionary $\Es = \{e_i\}_{i\in\Z}$ fulfill the \emph{Neumann ERC}
  \begin{equation}\label{eq_suf}
    \sup\limits_{i\in I} \sum\limits_{\sumstack{j\in I}{j \neq i}} |\ip{d_i}{d_j}|
    + \sup\limits_{i\in I^\complement} \sum\limits_{j\in I} |\ip{d_i}{d_j}|
    <1,
  \end{equation}
   then OMP with its parameter $\varepsilon$ set to 0 recovers $\alpha$.
\end{proposition}
\noindent
The proof uses a Neumann series estimate for $P_I D^\ast D P_I$---this
clarifies the term ``Neumann'' ERC. The proof is contained in~\cite{Gribonval2008}.

\begin{remark}
  Obviously the Neumann ERC~(\ref{eq_suf}) is not necessary for exact recovery.
  A demonstrative example can be found in $\R^4$ with the signal
  $v=(1,1,1,0)^\top$ and the unit-normed dictionary
  $\Ds=\{d_1:=(1,0,0,0)^\top,d_2:=2^{-1/2}(1,1,0,0)^\top,d_3:=2^{-1/2}(1,0,1,0)^\top,d_4:=(0,0,0,1)^\top\}$.
  Here with $I=\{1,2,3\}$ and $I^\complement=\{4\}$ we get
  \[
    |\ip{d_1}{d_4}| + |\ip{d_2}{d_4}| + |\ip{d_3}{d_4}| = 0
  \]
  but
  \[
    |\ip{d_1}{d_2}| + |\ip{d_1}{d_3}| = \sqrt{2} > 1,
  \]
  hence the Neumann ERC is not fulfilled. The ERC~(\ref{eq_ERC}) is nevertheless fulfilled since in that
  case $\norm{(D P_I)^\dagger d_4}_{\ell^1}=0$.
 OMP will then recover exactly, as one could expect by considering that just $\{d_1,d_2,d_3\}$ span the $\R^3$.

  This counter-example may be generalized by considering $I\subset\Z$ such that
  \[
    \sup\limits_{i\in I^\complement} \sum\limits_{j\in I} |\ip{d_i}{d_j}| = 0
    \quad \mbox{ and } \quad
    \sup\limits_{i\in I} \sum\limits_{\sumstack{j\in I}{j \neq i}} |\ip{d_i}{d_j}| \geq 1.
  \]
  Here the Neumann ERC fails but for any signal with support $I$
  OMP will recover exactly since the atoms
  $d_i$, $i\in I$, and $d_j$, $j\in I^\complement$, are uncorrelated
  and OMP never chooses an atom twice.

\end{remark}

\begin{remark}
  The sufficient conditions for WOMP with weakness parameter $\omega\in(0,1]$ are
  \[
    \sup\limits_{d\in \Ds(I^\complement)} \norm{(D P_I)^\dagger d}_{\ell^1}
    <\omega
  \]
  and
  \[
    \sup\limits_{i\in I} \sum\limits_{\sumstack{j\in I}{j \neq i}} |\ip{d_i}{d_j}|
    + \frac{1}{\omega} \, \sup\limits_{i\in I^\complement} \sum\limits_{j\in I} |\ip{d_i}{d_j}|
    <1,
  \]
  according to theorem~\ref{theorem_ERC} and
  proposition~\ref{theorem_suf}, respectively.
  They are proved analogously to the OMP case---same as all other following WOMP results.
\end{remark}

Usually for sparse approximation problems the behavior of the dictionary is characterized
as follows.
\begin{definition}
  Let $\Fs:=\{f_i\}_{i\in\Z}$ be a dictionary. Then the corresponding \emph{coherence parameter} $\mu$
  and \emph{cumulative coherence} $\mu_1(m)$ for a positive integer $m$ are
  defined as
  \[
    \mu := \sup\limits_{i\neq j}
    |\ip{f_i}{f_j}|
  \]
  and
  \[
    \mu_1(m) := \sup\limits_{\sumstack{\Lambda\subset\Z}{|\Lambda|=m}}
    \sup\limits_{i\notin\Lambda}\,
    \sum\limits_{j\in\Lambda}
    |\ip{f_i}{f_j}|
  \]
  respectively. Note that $\mu_1(1) = \mu$ and
  $\mu_1(m) \leq m\mu$ for all $m\in\N$.
\end{definition}

Since $\sup_{i\in I} \sum_{j\in I,j \neq i} |\ip{d_i}{d_j}| \leq \mu_1(N-1)$
and $\sup_{i\in I^\complement} \sum_{j\in I} |\ip{d_i}{d_j}| \leq \mu_1(N)$
we
get another condition in terms of the cumulative coherence,
which is even weaker than the Neumann ERC:
\begin{proposition}[Tropp~\cite{Tropp2004greed}]\label{theorem_coherence}
  Let $\alpha \in \ell^0$ with $\supp \alpha = I$,
  $u = \sum_{i\in\Z} \alpha_i e_i$ be the source and $v = K u$ the measured signal.
  If the operator $K:B \to H$
  and the dictionary $\Es = \{e_i\}_{i\in\Z}$ lead to a dictionary
  $\Ds$ which  fulfills the condition
  \begin{equation}\label{eq_coherence}
    \mu_1(N-1)
    + \mu_1(N)
    <1,
  \end{equation}
  then OMP with its parameter $\varepsilon$ set to 0 recovers $\alpha$.
\end{proposition}

Remark, that the condition in proposition~\ref{theorem_coherence} for
ill-posed inverse problems might be unsuitable, since the typically compact operator
causes that the coherence parameter $\mu$ probably is close to one.
Therefore the cumulative coherence can grow large with increasing support.

\begin{remark}\label{remark_bp}
  Another major approach for solving sparse approximation problems is the basis pursuit~(BP).
  Here one solves the convex optimization problem
  \[
    \min\limits_{\alpha\in\ell^2} \norm{\alpha}_{\ell^1}
    \quad \mbox{subject to} \quad
    K \textstyle\sum \alpha_i e_i = v.
  \]
  This idea is closely related to
  Tikhonov regularization with sparsity constraint.
  Here the basic idea is to minimize
  least squares with $\ell^1$-penalty,
  \[
    \min\limits_{\alpha\in\ell^2} \norm{K \textstyle\sum \alpha_i e_i - v}^2_H + \gamma \norm{\alpha}_{\ell^1}.
  \]
  In~\cite{Tropp2004greed} it is shown that the ERC~(\ref{eq_ERC}) also
  ensures the exact recovery by means of BP.
  Since the proposition~\ref{theorem_suf} and proposition~\ref{theorem_coherence}
  are estimates for the ERC~(\ref{eq_ERC}) and the proofs do not take into account
  any properties of the OMP algorithm these results hold here, too.
\end{remark}

\section{Exact Recovery in the Presence of Noise}\label{sec_epsERC}
In~\cite{donoho2004stablesparse}, Donoho, Elad and Temlyakov transfer Tropp's result~\cite{Tropp2004greed}
to noisy signals. They derive a condition for exact recovery in terms of 
the coherence parameter $\mu$ of a dictionary.
This condition is---just as remarked in~\cite{donoho2004stablesparse}---an obvious weaker condition.
As already mentioned, in particular for ill-posed problems this condition
is too restrictive.
In the following we will give exact recovery conditions in the presence of noise
which are closer to the results of theorem~\ref{theorem_ERC} and proposition~\ref{theorem_suf}.

Assume that instead of exact data $v=Ku \in H$ only a noisy version 
\[
  v^\varepsilon = v+\eta = Ku+\eta
\]
with noise level $\norm{v-v^\varepsilon} = \norm{\eta} \leq\varepsilon$ can be observed.
Now, the OMP has to stop as soon as the representation error
$r^k$ is smaller or equal to the noise level $\varepsilon$, i.e. if $\varepsilon \geq \norm{r^k}$.

\begin{theorem}[ERC in the Presence of Noise]\label{theorem_ERCnoise}
  Let $\alpha \in \ell^0$ with $\supp \alpha = I$.
  Let $u = \sum_{i\in\Z} \alpha_i e_i$ be the source and 
  $v^\varepsilon = K u + \eta$ the noisy data with noise level $\norm{\eta}\leq \varepsilon$ and
  noise-to-signal-ratio
  \[
    \nsr := \frac{\sup\limits_{i\in\Z} |\ip{\eta}{d_i}|}
              {\min\limits_{i\in I} |\alpha_i| \norm{K e_i}}.
  \]
  If  the operator $K:B \to H$
  and the dictionary $\Es = \{e_i\}_{i\in\Z}$ fulfill the \emph{Exact Recovery Condition in Presence of Noise ($\varepsilon$ERC)}
  \begin{equation}\label{eq_ERCnoise}
    \sup\limits_{d\in \Ds(I^\complement)} \norm{(D P_I)^\dagger d}_{\ell^1}
    <1 - 2 \, \nsr 
    \frac{1}{1-\sup\limits_{i\in I} \sum\limits_{\sumstack{j\in I}{j \neq i}} |\ip{d_i}{d_j}|},
  \end{equation}
 	and $\sup\limits_{i\in I} \sum\limits_{\sumstack{j\in I}{j \neq i}} |\ip{d_i}{d_j}|<1$,
  then OMP recovers the support $I$ of $\alpha$ exactly.
\end{theorem}

\begin{proof}
  We prove the $\varepsilon$ERC analogously to~\cite[theorem 3.1]{Tropp2004greed}
  by induction.
  Assume that OMP recovered the correct patterns in the first $k$ steps, i.e.
  \[
    \widehat{u}^k = \sum\limits_{i\in I^k} \widehat{\alpha}_i^k e_i,
  \]
  with $I^k\subset I$. Then we get for the residual
  \begin{align*}
    r^k & := v^\varepsilon - K \widehat{u}^k
    = v + \eta - K \widehat{u}^k
    = K \Big(\sum\limits_{i \in I} (\alpha_i - \widehat{\alpha}_i^k) e_i \Big) + \eta \\
    & = \sum\limits_{i \in I} \norm{Ke_i}(\alpha_i - \widehat{\alpha}_i^k) \, d_i + \eta,
  \end{align*}
  hence the noiseless residual $s^k := \sum \norm{Ke_i}(\widehat{\alpha}_i^k - \alpha_i) d_i$
  has support $I$.
  The correlation $|\ip{r^k}{d_i}|$, $i\in\Z$, can be estimated from below and above respectively via
  \[
    |\ip{r^k}{d_i}| = |\ip{s^k+\eta}{d_i}| \gtreqqless |\ip{s^k}{d_i}| \mp |\ip{\eta}{d_i}|.
  \]
  Hence with
  \[
    \sup\limits_{i\in I^\complement} |\ip{r^k}{d_i}| \leq \sup\limits_{i\in I^\complement} |\ip{s^k}{d_i}| + \sup\limits_{i\in\Z}|\ip{\eta}{d_i}|
  \]
  and
  \[
    \sup\limits_{i\in I} |\ip{s^k}{d_i}| - \sup\limits_{i\in\Z}|\ip{\eta}{d_i}| \leq \sup\limits_{i\in I} |\ip{r^k}{d_i}|
  \]
  we get the condition
  \[
    \norm{P_{I^\complement} D^\ast s^k}_{\ell^\infty} + \sup\limits_{i\in\Z}|\ip{\eta}{d_i}| < \norm{P_{I} D^\ast s^k}_{\ell^\infty} - \sup\limits_{i\in\Z}|\ip{\eta}{d_i}|,
  \]
  which ensures a right choice in the $(k+1)$-th step. 
  Since $(P_I D^\ast)^\dagger P_I D^\ast$ is the orthogonal projection onto $\Ds(I)$
  and $\supp s^k = I$ we can write
  \[
    s^k = (P_I D^\ast)^\dagger P_I D^\ast s^k.
  \]
  With this identity we get the sufficient condition for OMP
  in presence of noise
  \[
    \frac
    {\norm{P_{I^\complement} D^\ast (P_I D^\ast)^\dagger P_I D^\ast s^k}_{\ell^\infty}}
    {\norm{P_{I} D^\ast s^k}_{\ell^\infty}}
    < 1 - 2\,\frac{\sup\limits_{i\in\Z}|\ip{\eta}{d_i}|}{\norm{P_I D^\ast s^k}_{\ell^\infty}}.
  \]
  Consequently, since 
  $\norm{P_{I^\complement} D^\ast (P_I D^\ast)^\dagger P_I D^\ast s^k}_{\ell^\infty}
  \leq \norm{P_{I^\complement} D^\ast (P_I D^\ast)^\dagger}_{{\ell^\infty},{\ell^\infty}}
  \,\norm{P_{I} D^\ast s^k}_{\ell^\infty}$
  and with the definition of the adjoint operator $D^\ast$,
  we get the equivalent sufficient conditions for a correct choice in the $(k+1)$-th step
  \[
    \norm{P_{I^\complement} D^\ast (P_I D^\ast)^\dagger}_{{\ell^\infty},{\ell^\infty}}
    = \norm{(D P_I)^\dagger D P_{I^\complement}}_{{\ell^1},{\ell^1}}
    < 1 - 2\,\frac{\sup\limits_{i\in\Z}|\ip{\eta}{d_i}|}{\norm{P_I D^\ast s^k}_{\ell^\infty}}.
  \]
  Obviously, on the one hand we get
  \[
    \sup\limits_{i\in I^\complement} \norm{(D P_I)^\dagger d_i}_{\ell^1}
    \leq \sup\limits_{\sumstack{\norm{\beta}_{\ell^1}=1}{\supp \beta = I^\complement}} \norm{(D P_I)^\dagger \sum\limits_{i\in\Z} \beta_i d_i}_{\ell^1}
    = \norm{(D P_I)^\dagger D P_{I^\complement}}_{{\ell^1},{\ell^1}}
  \]
  and on the other hand, since $(DP_I)^\dagger$ is linear, we get
  \begin{align*}
    \sup\limits_{\sumstack{\norm{\beta}_{\ell^1}=1}{\supp \beta = I^\complement}} \norm{(D P_I)^\dagger \sum\limits_{i\in\Z} \beta_i d_i}_{\ell^1}
    & \leq \sup\limits_{\sumstack{\norm{\beta}_{\ell^1}=1}{\supp \beta = I^\complement}} \sum\limits_{i\in\Z} |\beta_i|
      \sup\limits_{d\in \Ds(I^\complement)} \norm{(D P_I)^\dagger d}_{\ell^1}\\
    & = \sup\limits_{d\in \Ds(I^\complement)} \norm{(D P_I)^\dagger d}_{\ell^1}.
  \end{align*}
  This shows that 
  \[
    \norm{(D P_I)^\dagger D P_{I^\complement}}_{{\ell^1},{\ell^1}}
    = \sup\limits_{d\in \Ds(I^\complement)} \norm{(D P_I)^\dagger d}_{\ell^1}
    < 1 - 2\,\frac{\sup\limits_{i\in\Z}|\ip{\eta}{d_i}|}{\norm{P_I D^\ast s^k}_{\ell^\infty}}
  \]
  is another equivalent condition for exact recovery.
  The last thing we have to afford to finish the proof is an estimation for the term 
  $\norm{P_I D^\ast s^k}_{\ell^\infty}$ from below.

  In the first step this is easy, since $r^0 = v^\varepsilon$ resp.~$s^0=v$
  with $v=Ku=K\sum_{i\in I} \alpha_i e_i$. With that we get
  \begin{align*}
    \norm{P_I D^\ast s^0}_{\ell^\infty}
    & = \norm{P_I D^\ast v}_{\ell^\infty}
      = \sup\limits_{j\in I} |\ip{v}{d_j}|
    = \sup\limits_{j\in I} \Big|\sum\limits_{i\in I} \alpha_i \norm{Ke_i} \ip{d_i}{d_j} \Big|\\
    & \geq \Big|\sum\limits_{i\in I} \alpha_i \norm{Ke_i} \ip{d_i}{d_l} \Big| 
    	\geq |\alpha_l| \, \norm{Ke_l} \big(1-\sup\limits_{i\in I} \sum\limits_{\sumstack{j\in I}{j \neq i}} |\ip{d_i}{d_j}|\big)
  \end{align*}
  for all $l\in I$, hence in particular
  \[
    \norm{P_I D^\ast v}_{\ell^\infty}
    \geq 
    \min\limits_{i\in I} |\alpha_i| \, \norm{Ke_i} \big(1-\sup\limits_{i\in I} \sum\limits_{\sumstack{j\in I}{j \neq i}} |\ip{d_i}{d_j}|\big).
  \]

  To prove this for general $k$ we successively apply this estimation.
  Here, again, we get
  \begin{align*}
    \norm{P_I D^\ast s^k}_\infty
    & = \sup\limits_{j\in I} |\ip{s^k}{d_j}|
    = \sup\limits_{j\in I} \Big|\sum\limits_{i\in I} (\widehat{\alpha}_i^k - \alpha_i) \norm{Ke_i} \ip{d_i}{d_j}\Big| \\
    & \geq \Big|\sum\limits_{i\in I} (\widehat{\alpha}_i^k - \alpha_i) \norm{Ke_i} \ip{d_i}{d_l} \Big|
    \geq |\alpha_l| \, \norm{Ke_l} \big(1-\sup\limits_{i\in I} \sum\limits_{\sumstack{j\in I}{j \neq i}} |\ip{d_i}{d_j}|\big)
  \end{align*}
  for all $l\in I$, $l\notin I^k$, hence, since OMP never chooses an atom twice, in particular
  \[
    \norm{P_I D^\ast s^k}_{\ell^\infty}
    \geq
    \min\limits_{i\in I} |\alpha_i| \, \norm{Ke_i} \big(1-\sup\limits_{i\in I} \sum\limits_{\sumstack{j\in I}{j \neq i}} |\ip{d_i}{d_j}|\big).
  \]
\end{proof}

In particular, to ensure the $\varepsilon$ERC~(\ref{eq_ERCnoise}) one has necessarily
for the noise-to-signal-ratio $\nsr < 1/2$.
For a small noise-to-signal-ratio the $\varepsilon$ERC~(\ref{eq_ERCnoise})
approximates the ERC~(\ref{eq_ERC}).
A rough upper bound for $\sup_{i\in\Z} |\ip{\eta}{d_i}|$ 
is $\varepsilon$ and hence, one may use
\[
  \nsr \leq \frac{\varepsilon}{\min\limits_{i\in I} |\alpha_i| \norm{K e_i}}.
\]

Similar to  the noiseless case,
the $\varepsilon$ERC~(\ref{eq_ERCnoise}) is hard to evaluate.
Analogously to section~\ref{sec_ERC} we now give a
weaker sufficient recovery condition that depends on inner products of the dictionary atoms.
It is proved analogously to proposition~\ref{theorem_suf}.
\begin{proposition}[Neumann ERC in the Presence of Noise]\label{theorem_sufnoise}
  Let $\alpha \in \ell^0$ with $\supp \alpha = I$.
  Let $u = \sum_{i\in\Z} \alpha_i e_i$ be the source and
  $v^\varepsilon = K u + \eta$ the noisy data with noise level $\norm{\eta}\leq \varepsilon$ and
  noise-to-signal-ratio $\nsr < 1/2$.
  If the operator $K:B \to H$
  and the dictionary $\Es = \{e_i\}_{i\in\Z}$ fulfill
  the \emph{Neumann $\varepsilon$ERC}
  \begin{equation}\label{eq_sufnoise}
    \sup\limits_{i\in I} \sum\limits_{\sumstack{j\in I}{j \neq i}} |\ip{d_i}{d_j}|
    + \sup\limits_{i\in I^\complement} \sum\limits_{j\in I} |\ip{d_i}{d_j}|
    <1 - 2\,\nsr,
  \end{equation}
  then OMP recovers the support $I$ of $\alpha$ exactly.
\end{proposition}
\noindent

\begin{remark}
  The according suffient conditions for exact
  recovery with WOMP with weakness parameter $\omega\in(0,1]$
  for the case of noisy data with noise-to-signal-ratio $\nsr < \omega /2$
  are
  \[
    \sup\limits_{d\in \Ds(I^\complement)} \norm{(D P_I)^\dagger d}_{\ell^1}
    <\omega - 2 \, \nsr
    \frac{1}{1-\sup\limits_{i\in I} \sum\limits_{\sumstack{j\in I}{j \neq i}} |\ip{d_i}{d_j}|},
  \]
  and
  \[
    \sup\limits_{i\in I} \sum\limits_{\sumstack{j\in I}{j \neq i}} |\ip{d_i}{d_j}|
    + \frac{1}{\omega} \sup\limits_{i\in I^\complement} \sum\limits_{j\in I} |\ip{d_i}{d_j}|
    <1 - \frac{2\nsr}{\omega},
  \]
  analog to theorem~\ref{theorem_ERCnoise} and
  proposition~\ref{theorem_sufnoise}, respectively.
\end{remark}

Same as for the noiseless case we can give another even weaker condition in terms of the cumulative coherence:
\begin{proposition}\label{theorem_coherencenoise}
  Let $\alpha \in \ell^0$ with $\supp \alpha = I$.
  Let $u = \sum_{i\in\Z} \alpha_i e_i$ be the source and
  $v^\varepsilon = K u + \eta$ the noisy data with noise level $\norm{\eta}\leq \varepsilon$ and
  noise-to-signal-ratio $\nsr < 1/2$.
  If the operator $K:B \to H$
  and the dictionary $\Es = \{e_i\}_{i\in\Z}$ lead to dictionary $\Ds$
  which fulfills the condition
  \begin{equation}\label{eq_coherencenoise}
    \mu_1(N-1)
    + \mu_1(N)
    < 1 - 2\,\nsr,
  \end{equation}
  then OMP recovers the support $I$ of $\alpha$ exactly.
\end{proposition}

Remark that theorem~\ref{theorem_ERCnoise} and
proposition~\ref{theorem_sufnoise} just ensure the correct support
$I$. The following simple proposition shows that the reconstruction
error is of the order of the noise level.
\begin{proposition}[Error bounds for OMP in presence of noise]
  \label{prop_error_bound_omp}
  Let $\alpha \in \ell^0$ with $\supp \alpha = I$.
  Let $u = \sum_{i\in\Z} \alpha_i e_i$ be the source and
  $v^\varepsilon$ the noisy data
  with noise level $\norm{v^\varepsilon - v} \leq \varepsilon$ and
  noise-to-signal-ratio $\nsr < 1/2$.
  If the $\varepsilon$ERC is fulfilled then
  there exists a constant $C>0$ such that 
  for the approximative solution $\widehat{\alpha}$
  determined by OMP
  the following error bound holds,
  \[
    \norm{\widehat{\alpha} - \alpha}_{\ell^1}
    \leq C \varepsilon.
  \]
\end{proposition}

\begin{proof}
  Since the $\varepsilon$ERC is fulfilled OMP recovered the correct support $I$, i.e.
  \[
    \widehat{\alpha} = 
    \argmin \Set{\norm{v^\varepsilon - \textstyle\sum_{i\in I}\check\alpha_i Ke_i}}{\check\alpha \in \ell^2(I)}.
  \]
  With the help of the operator $A:\ell^1(I)\to H$ defined by $A\alpha = \sum_{i\in I}\alpha_i Ke_i$ this is equivalently written as
  \[
   A^* A\widehat{\alpha} = A^* v^\varepsilon.
  \]
  Note that $A^*A:\ell^2(I)\to \ell^2(I)$ is just the matrix
  \[
  (A^* A)_{i,j} = \ip{Ke_i}{Ke_j}.
  \]
  For the error we get
  \[
    \norm{\widehat{\alpha} - \alpha}_{\ell^1}
    =  \norm{A^\dagger(v^\varepsilon - v)}_{\ell^1}
    \leq \norm{A^\dagger}_{H,\ell^1} \, \norm{(v^\varepsilon - v)}_H
    = C\varepsilon.
  \]
\end{proof}

\begin{remark}\label{remark_bp_noise}
  We remark again on an exact recovery condition for BP.
  Unlike the section~\ref{sec_ERC} where the results
  can be transfered to BP, see remark~\ref{remark_bp}, 
  this is not possible for
  the presence of noise:
  To prove theorem~\ref{theorem_ERCnoise}
  we used properties of the OMP algorithm which are not valid
  for BP.

  For the case of noisy data $v^\varepsilon$
  in~\cite{donoho2004stablesparse} an exact recovery condition for BP is derived.
  This condition depends on the coherence parameter $\mu$.
  Since in this paper the focus is on the \emph{greedy} solution of inverse problems
  we give up on deriving stronger results for BP which are closer to
  the results of this section.
\end{remark}

\section{Resolution Bounds for Mass Spectrometry}\label{sec_massspectro}

Granted, to apply the Neumann conditions of
proposition~\ref{theorem_suf} and proposition~\ref{theorem_sufnoise},
respectively, one has to know the support $I$.
In this case there would be no need to apply OMP---one may just
solve the restricted least squares problem, i.e.~project onto $\spa\Es(I)$.
For deconvolution problems, however, with certain prior knowledge the Neumann ERC~(\ref{eq_suf})
resp.~Neumann $\varepsilon$ERC~(\ref{eq_sufnoise})
are easier to evaluate than the ERC~(\ref{eq_ERC}) resp.~$\varepsilon$ERC~(\ref{eq_ERCnoise}) especially when the support $I$ is not known exactly.
In the following we will use the weaker conditions exemplarily
with impulse trains convolved with Gaussian kernel
as e.g.~occurs in mass spectrometry, cf.~\cite{klann2007shrink_vs_deconv}.

\subsection{Analysis}
In mass spectrometry the source $u$ is given---after
simplification---as sum of Dirac peaks at integer positions $i\in\Z$,
\[
  u = \sum\limits_{i\in \Z} \alpha_i \, \delta(\cdot - i),
\]
with $|\supp \alpha| = |I| = N$. Since the measuring procedure is influenced by
Gaussian noise the measured data can be modeled by a convolution operator $K$ with 
Gaussian kernel
\begin{equation}
  \label{eq:gaussian_kernel}
  \kappa(x) = \frac{1}{\pi^{1/4}\sigma^{1/2}} \exp \Big( -\frac{x^2}{2\sigma^2} \Big),
\end{equation}
i.e.~the operator under consideration is $Ku = \kappa \ast u$.  As Banach
space $B$ we may use the space $\mathcal{M}$ of regular Borel measures
on $\R$ (which contains impulse trains if the coefficients $\alpha_i$
are summable) and as Hilbert space $H$ the space $L^2(\R)$.
We form the dictionary $\Es$ of Dirac peaks at integer positions and hence, we have $\Ds = \{\delta(\cdot - i) \ast \kappa\} = \{\kappa(\cdot-i)\}$,
since $\norm{\kappa(\cdot-i)}_{L^2} = 1$, $i\in\Z$.

To verify the Neumann ERC~(\ref{eq_suf})
and Neumann $\varepsilon$ERC~(\ref{eq_sufnoise})
respectively, we need the autocorrelation of two atoms $\kappa(\cdot-i)$
and $\kappa(\cdot-j)$. In $L^2(\R,\R)$ it arises as
\begin{equation}
  \label{eq:correlation_two_gaussians}
  \ip{\kappa(\cdot-i)}{\kappa(\cdot-j)}_{L^2}
  = \int\limits_{\R} \; \tfrac{1}{\sqrt{\pi}\sigma} \exp \big( -\tfrac{(x-i)^2}{2\sigma^2} \big)
    \exp \big( -\tfrac{(x-j)^2}{2\sigma^2} \big) \rd x
  = \exp\big(- \tfrac{(i-j)^2}{4\sigma^2} \big),
\end{equation}
which is positive and monotonically decreasing in the distance $|i-j|$.
If we additionally assume that the peaks of any source $u$ have the minimal distance
\[
  \rho := \min\limits_{i,j\in\supp\alpha} |i-j|,
\]
then w.l.o.g. we can estimate the sums of correlations from above as
follows.
%
%
For $\rho\in\N$ we get for the correlations of support atoms
\begin{align*}
  \sup\limits_{i\in I} \sum\limits_{\sumstack{j\in I}{j \neq i}} |\ip{d_i}{d_j}|
      & \leq 2 \sum\limits_{j=1}^{\lfloor N/2 \rfloor} \ip{\kappa}{\kappa(\cdot-j\rho)}
      = 2 \sum\limits_{j=1}^{\lfloor N/2 \rfloor} \exp\big(- \tfrac{(j\rho)^2}{4\sigma^2} \big).
\intertext{For the correlations of support atoms and non-support atoms we have to
distinguish between two cases for $\rho$. For $\rho \geq 2$ we get}
  \sup\limits_{i\in I^\complement} \sum\limits_{j\in I} |\ip{d_i}{d_j}|
      & \leq \sup\limits_{1\leq i<\rho} \sum\limits_{j=-\lfloor N/2 \rfloor}^{\lfloor N/2 \rfloor} \ip{\kappa(\cdot-i)}{\kappa(\cdot-j\rho)}
      = \sup\limits_{1\leq i<\rho} \sum\limits_{j=-\lfloor N/2 \rfloor}^{\lfloor N/2 \rfloor} \exp\big(- \tfrac{(i-j\rho)^2}{4\sigma^2} \big)
\intertext{and for $\rho=1$}
  \sup\limits_{i\in I^\complement} \sum\limits_{j\in I} |\ip{d_i}{d_j}|
      & \leq 2 \sum\limits_{j=1}^{\lfloor N/2 \rfloor+1} \ip{\kappa}{\kappa(\cdot-j)}
      = 2 \sum\limits_{j=1}^{\lfloor N/2 \rfloor+1} \exp\big(- \tfrac{j^2}{4\sigma^2} \big).
\end{align*}

With that we can formulate the Neumann ERC and the
Neumann $\varepsilon$ERC for Dirac peaks convolved with Gaussian kernel.
\begin{proposition}\label{theorem_gauss}
  An estimation from above for the ERC (i.e. $\nsr=0$) and
  $\varepsilon$ERC (i.e. $0<\nsr<\frac{1}{2}$)
  for Dirac peaks convolved with Gaussian kernel is for $\rho\geq 2$
  \begin{align*}
    2 \sum\limits_{j=1}^{\lfloor N/2 \rfloor} \exp\Big(- \frac{(j\rho)^2}{4\sigma^2} \Big)
    + \sup\limits_{1\leq i<\rho} \sum\limits_{j=-\lfloor N/2 \rfloor}^{\lfloor N/2 \rfloor} \exp\Big(- \frac{(i-j\rho)^2}{4\sigma^2} \Big)
    < 1 - 2\,\nsr,
  \end{align*}
  and for $\rho=1$
  \[
    2 \sum\limits_{j=1}^{\lfloor N/2 \rfloor} \exp\Big(- \frac{j^2}{4\sigma^2} \Big)
    + 2 \sum\limits_{j=1}^{\lfloor N/2 \rfloor+1} \exp\Big(- \frac{j^2}{4\sigma^2} \Big)
    < 1 - 2\,\nsr.
  \]
  This means that we are able to recover the support of the impulse
  train with OMP exactly from the convolved data if the above
  conditions are fulfilled.
\end{proposition}

\begin{remark}
  Remark that the case $\rho=1$ of proposition~\ref{theorem_gauss}
  coincides more or less with the recovery condition in terms of the
  cumulative coherence, since for odd $N$ we get
  $\mu_1(N) = 2 \, \sum_{j=1}^{N/2} \exp(- j^2 / (4\sigma^2))$.
  Summing up just over a subset of $\rho\Z:=\set{j\in\Z}{j/\rho \in\Z}$ is not 
  a feasible estimation, since for the support $I$
  we allow any point $i\in\Z$ and not only atoms of the sub-dictionary
  $\Ds(\rho\Z)$. This turns out to be the main disadvantage of the coherence condition:
  It does not distinguish between support and non-support atoms
  and is hence in some applications a clearly weaker estimation.
\end{remark}

\begin{remark}
  If the cardinality of the support $N$ is unknown one could replace
  the finite sums by infinite sums.  Obviously these sums exist since
  the geometric series is a majorizing series.  
  With $\iota$ representing the imaginary unit
  they can be expressed
  in terms of the Jacobi theta function of the third kind,
  $\jacobi(z,q) := \sum_{j=-\infty}^\infty q^{j^2} \exp (2 j \iota z)$.
\end{remark}

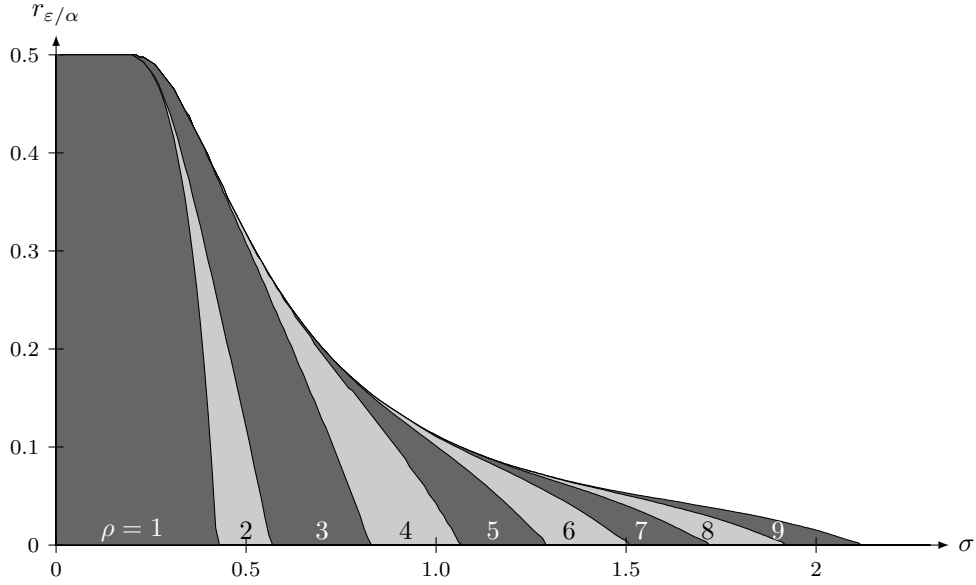
\begin{figure}[ht]
  \begin{tikzpicture}[xscale=5,yscale=13]
    \filldraw[fill=black!60!white] plot  file {./ERC_gauss9.dat} |- (0,0) |- (0,0.5);
    \filldraw[fill=black!20!white] plot  file {./ERC_gauss8.dat} |- (0,0) |- (0,0.5);
    \filldraw[fill=black!60!white] plot  file {./ERC_gauss7.dat} |- (0,0) |- (0,0.5);
    \filldraw[fill=black!20!white] plot  file {./ERC_gauss6.dat} |- (0,0) |- (0,0.5);
    \filldraw[fill=black!60!white] plot  file {./ERC_gauss5.dat} |- (0,0) |- (0,0.5);
    \filldraw[fill=black!20!white] plot  file {./ERC_gauss4.dat} |- (0,0) |- (0,0.5);
    \filldraw[fill=black!60!white] plot  file {./ERC_gauss3.dat} |- (0,0) |- (0,0.5);
    \filldraw[fill=black!20!white] plot  file {./ERC_gauss2.dat} |- (0,0) |- (0,0.5);
    \filldraw[fill=black!60!white] plot  file {./ERC_gauss1.dat} |- (0,0) |- (0,0.5);
    \draw[white] (0.2, 0.015) node {$\rho=1$};
    \draw[black] (0.5, 0.015) node {2};
    \draw[white] (0.7, 0.015) node {3};
    \draw[black] (0.92, 0.015) node {4};
    \draw[white] (1.15, 0.015) node {5};
    \draw[black] (1.35, 0.015) node {6};
    \draw[white] (1.54, 0.015) node {7};
    \draw[black] (1.715, 0.015) node {8};
    \draw[white] (1.9, 0.015) node {9};
    \foreach \x in {0,0.5,1.0,1.5,2}
      \draw (\x,0.008) -- (\x,-0.008) node[below] {\footnotesize ${\x}$};
    \foreach \y in {0,0.1,0.2,0.3,0.4,0.5}
      \draw (0.02,\y) -- (-0.02,\y) node[left] {\footnotesize ${\y}$};
    \begin{scope}[>=latex]
      \draw[->] (0,0) -- (2.35,0) node[right] {$\sigma$};
      \draw[->] (0,0) -- (0,0.52) node[above] {$\nsr$};
    \end{scope}
  \end{tikzpicture}
  \caption{$\varepsilon$ERC for combinations of $\sigma$, $\rho$ and $\nsr$.}
  \label{bild_epsERC}
\end{figure}

The condition of proposition~\ref{theorem_gauss} is plotted
for some combinations of $\sigma$, $\rho$ and $\nsr$
with unknown $N$
in figure~\ref{bild_epsERC}.
The colored areas describe the combinations
where the Neumann ERC is fulfilled.

Often for deconvolution problems  the autocorrelation of two atoms 
$|\ip{d(\cdot-i)}{d(\cdot-j)}|$ is not
monotonically decreasing in the distance $|i-j|$
and it obviously depends on the kernel $\kappa$.
However, if the correlation of two atoms can be estimated from above
via a monotonically decreasing function w.r.t.~an appropriate distance
then we can use a similar estimate.
We do this exemplarily for an oscillating kernel in section~\ref{sec_holography}
namely, for Fresnel-convolved characteristic functions as appear in
digital holography.

\begin{remark}
  \label{rem_continuous_omp}
  We remark on a possible fully continuous formulation of OMP.
  We assume that we are given some data
  \[
  v = \kappa * u = \sum\limits_{i\in\Z} \alpha_i \, \kappa(\cdot-x_i)
  \]
  and that we do not know the positions $x_i$. We allow our
  dictionary to be uncountable, i.e.~we search for peaks at every real
  number. Note that here $i\in\Z$ does not represent the
  set of possible positions for peaks but it is an index set for continuous positions
  $x_i\in\R$.

  In the first step of the matching pursuit we correlate $v$
  with $\kappa(\cdot-x)$ and take that $x$ which gives maximal correlation.
  In the special case of the Gausssian blurring
  kernel~(\ref{eq:gaussian_kernel}) this amounts in finding the maximum
  of the function
  \[
  f(x) = |\ip{v}{\kappa(\cdot-x)}| = \sum\limits_{i\in\Z} \alpha_i
  \ip{\kappa(\cdot-x_i)}{\kappa(\cdot-x)}.
  \]
  From~\eqref{eq:correlation_two_gaussians} we see that this is
  \[
  f(x) = \sum\limits_{i\in\Z} \alpha_i \exp\Big(-\frac{(x-x_i)^2}{4\sigma^2} \Big).
  \]
  It is clear that any maximum of $f$ is unlikely to be precisely
  at some of the $x_i$'s, albeit very close.  A detailed study of this
  effect goes beyond the scope of this paper and we present a simple
  example.
  
  Let us assume that we have two peaks, one at 0 and one at $x_1$:
  \begin{equation}
    \label{eq:example_continuous_mp}
    u = \alpha_0 \delta(\cdot) + \alpha_1\delta(\cdot-x_1).
  \end{equation}
  Moreover, we assume that $\alpha_0>\alpha_1$, i.e.~the peak in zero
  is higher. The matching pursuit will find the first peak at the
  maximum of the function
  \[
  f(x) = \alpha_0 \exp\Big(-\frac{x^2}{4\sigma^2}\Big) +
  \alpha_1\exp\Big(-\frac{(x-x_1)^2}{4\sigma^2}\Big)
  \]
  and hence at some root of
  \[
  f'(x) = -\tfrac{1}{2\sigma^2}\Big(\alpha_0 x
  \exp\big(-\tfrac{x^2}{4\sigma^2}\big) + \alpha_1
  (x-x_1)\exp\big(-\tfrac{(x-x_1)^2}{4\sigma^2}\big)\Big).
  \]
  The error, that the matching pursuit makes is hence the error
  $\varrho$ in the root of $f'$ near zero.  In
  figure~\ref{fig:sigma_delta} it is shown, how the root of $f'$ close
  to zero depends on the variance $\sigma$.  One observes that the
  error $\varrho$ is smaller than the variance $\sigma$ by some orders
  of magnitude.

  As a final remark we mention that we measured the error not in some
  norm but only the distance of the $\delta$-peaks. This corresponds
  to the so called \emph{Prokhorov-metric} which is a metrization for
  the weak-* convergence in measure space.
\end{remark}

  
  
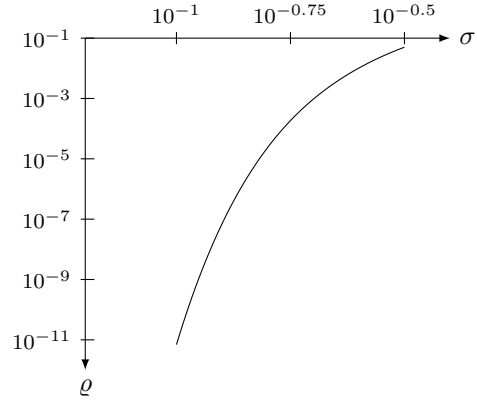
\begin{figure}
  \centering
  \begin{tikzpicture}[xscale=6,yscale=.4]
  \begin{scope}[>=latex]
    \draw[->] (-1.2,-1) -- (-.4,-1) node[right] {$\sigma$};
    \draw[->] (-1.2,-1) -- (-1.2,-12) node[below] {$\varrho$};  
  \end{scope}
    \foreach \x in {-1,-0.75,-0.5}
      \draw (\x,-1.2) -- (\x,-.8) node[above] {\footnotesize $10^{\x}$};
    \foreach \y in {-1,-3,...,-11}
      \draw (-1.18,\y) -- (-1.21,\y) node[left] {\footnotesize $10^{\y}$};
    \draw plot  file {./sigma_delta.dat};
  \end{tikzpicture}
  \caption{Error of the first step of the matching pursuit for the
    signal~(\ref{eq:example_continuous_mp}) with $\alpha_0=2$,
    $\alpha_1=1$ and $x_1=1$. The variable $\sigma$ is the variance of
    the Gaussian kernel and $\varrho$ is the position at which the
    matching pursuit locates the first peak.}
  \label{fig:sigma_delta}
\end{figure}

\subsection*{Numerical Examples}
We apply the Neumann $\varepsilon$ERC of proposition~\ref{theorem_gauss}
to simulated data of an isotope pattern.
Here the data consist of equidistant peaks with different heights.
In our example we use four peaks with a distance of $\rho=5$ and heights 
of $130$, $220$, $180$ and $90$, cf. the balls at the top of figure~\ref{bild_simulation}.
After convolving with Gaussian kernel with $\sigma=1.125$ we apply a
Poisson noise model. This is realistic,
because in mass spectrometry a finite number of particles is counted.

In the first example with low noise (mean and variance of 1.5 for regions without peaks)
the Neumann $\varepsilon$ERC is fulfilled and hence OMP recovered the
support exactly, see middle of figure~\ref{bild_simulation}.
However, the condition is restrictive: For the second example the signal
is disturbed with huge noise (mean and variance of 30 for regions without peaks)
and the Neumann $\varepsilon$ERC is not fulfilled.
Certainly, OMP recovered the support exactly,
see bottom of figure~\ref{bild_simulation}.

\begin{figure}[p]
  \begin{center}
  \includegraphics[width=0.7\linewidth]{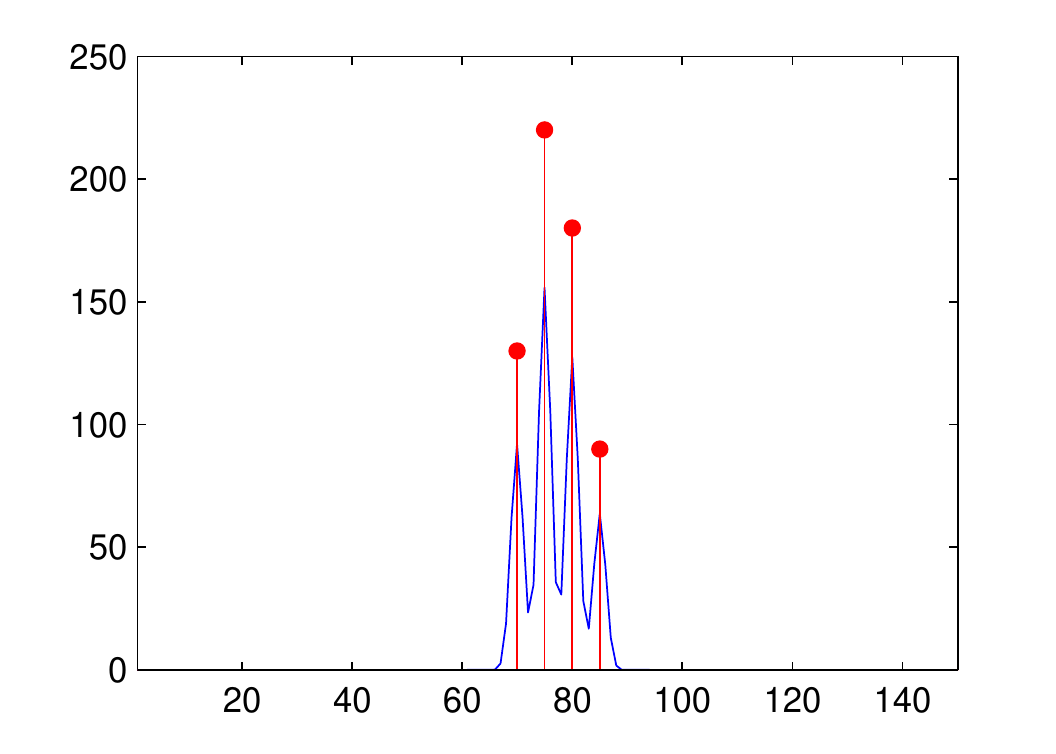}

  \includegraphics[width=0.7\linewidth]{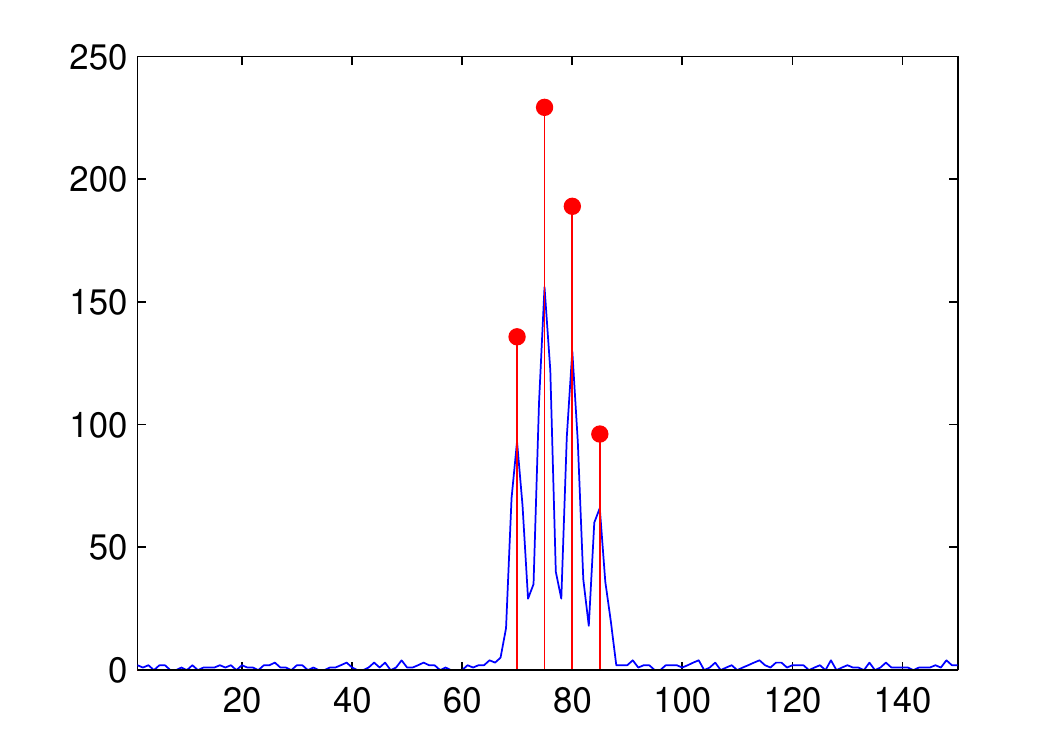}

  \includegraphics[width=0.7\linewidth]{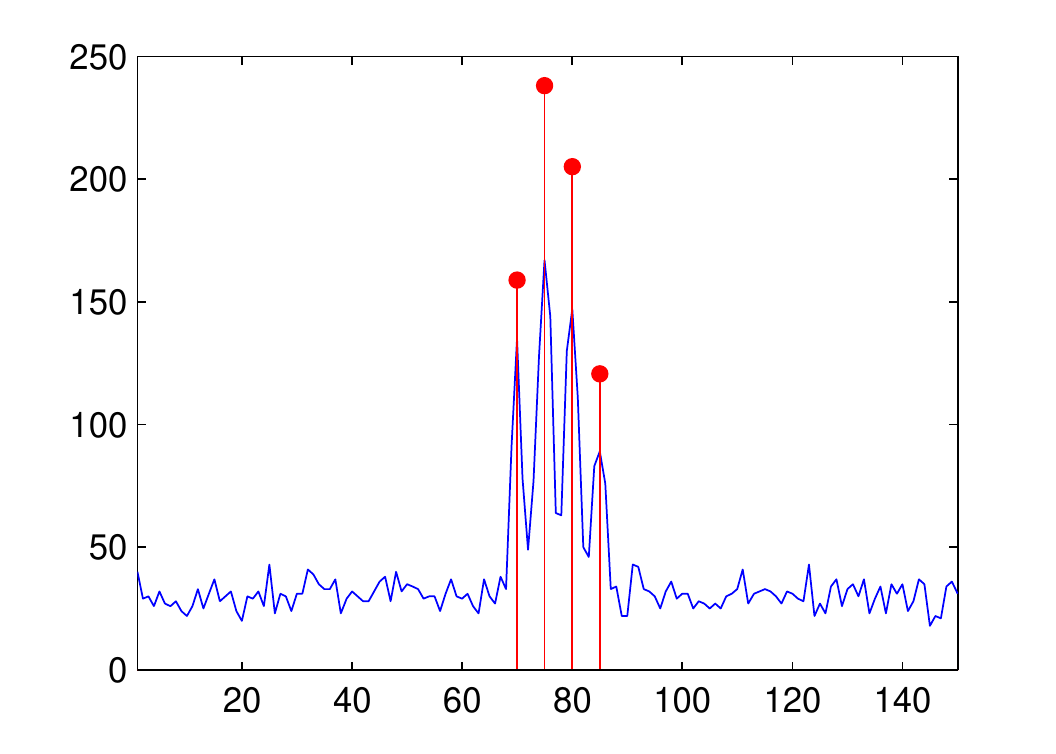}
  \end{center}
  \caption{Simulated isotope pattern. Top: support and Gaussian-convolved data without noise.
    Middle: low noise, Neumann $\varepsilon$ERC satisfied.
    Bottom: high noise, Neumann $\varepsilon$ERC not satisfied but still exact recovery possible.}
  \label{bild_simulation}
\end{figure}

\section{Resolution Bounds for Digital Holography}\label{sec_holography}
In digital holography, the data correspond to the diffraction patterns of the objects~\cite{goodman2005ifo,kreis2005hhi}.
Under Fresnel's approximation, diffraction can be modeled by a convolution with a ``chirp'' kernel.
In the context of holograms of particles~\cite{vikram1992pfh,hinsch1995tdp,hinsch2002hpi}, 
the objects can be considered opaque (i.e., binary) and the hologram recorded on the camera
 corresponds to the convolution of disks with Fresnel's chirp kernels. The measurement of
particle size and location therefore amounts to an inverse problem~\cite{soulez2007holography,soulez2007holography2}.

\subsection*{Analysis}
We consider the case of spherical particles
which is of significant interest in
applications such as fluid mechanics~\cite{zeff2003mir,meng2004hpi}.
We model the particles $j\in\{1,\ldots,N\}$
as opaque disks $B_r(\cdot-x_j,\cdot-y_j,\cdot-z_j)$
with center $(x_j,y_j,z_j)\in\R^3$, radius $r$ and disk orientation orthogonal to the optical axis $(Oz)$. 
Hence the source $u$ is given as a sum of characteristic functions
\[
  u = \sum\limits_{j=1}^N \alpha_j \, \chi_{B_r}(\cdot-x_j,\cdot-y_j,\cdot-z_j)
    =: \sum\limits_{j=1}^N \alpha_j \, \chi_j.
\]
The real values $\alpha_i$ are amplitude factors of the diffraction pattern
that in praxis depend on experimental parameters, cf.~\cite{Tyler1976,soulez2007holography}.

To an incident laser beam of (complex) amplitude $A_0$ and wavelength $\lambda$
the amplitude $A$ in the observation plane, i.e.~at depth $z=0$,
is well modeled by a bidimensional convolution $\convbi$ w.r.t.~$(x,y)$.
In the following $\iota$ represents the imaginary unit.
Then, with $\delta_{x_j,y_j}$ denoting Dirac's peak located at
$(x_j,y_j)$ and $h_{z_j}$ the Fresnel function,
\[
  h_{z_j}(x,y) = \frac{1}{\iota \lambda z_j} \exp \Big(\iota \frac{\pi}{\lambda z_j} R^2 \Big),
  \qquad 
  \mbox{with }
  R^2:=x^2+y^2,
\]
the amplitude $A:\R^2 \to \C$ arises as
\[
  A = A_0 \Big[1 - \sum\limits_{j=1}^N \alpha_j \,
  \big(\chi_j \convbi h_{z_j} \convbi \delta_{x_j,y_j}\big)\Big].
\]
Remark that $h_{z_j} \convbi \delta_{x_j,y_j}$ denotes the shifted Fresnel function.

One difficulty occurring at digital holography inverse problems is that in praxis only the absolute value
of $A$ can be measured by the detector and the phase gets lost. The measured intensity 
consequently arises as
\begin{align*}
  G = |A|^2 = |A_0|^2 \Big[ 1 & - 2 \sum\limits_{j=1}^N \alpha_j \,
    \big(\chi_j \convbi  \text{Re}(h_{z_j}) \convbi \delta_{x_j,y_j}\big)\\
  & + \sum\limits_{i=1}^N \sum\limits_{j=1}^N
    \alpha_i \big(\chi_i \convbi  h_{z_i} \convbi \delta_{x_i,y_i}\big)
    \alpha_j \big(\chi_j \convbi  h_{-z_j} \convbi \delta_{x_j,y_j}\big)
    \Big].
\end{align*}

Since the second term is dominant over the third one for $\chi$ small,
the intensity is classically linearized~\cite{Tyler1976,soulez2007holography}:
\begin{equation}\label{eq_fresnel}
  G \approx |A_0|^2 \Big[ 1 - 2 \sum\limits_{j=1}^N \alpha_j \,
    \big(\chi_j \convbi  \text{Re}(h_{z_j}) \convbi \delta_{x_j,y_j}\big)\Big].
\end{equation}

Analogously to section~\ref{sec_massspectro} we will next derive the Neumann ERC and the Neumann
ERC in presence of noise for the operator equation~(\ref{eq_fresnel}),
$\sum \alpha_j \, \chi_j \mapsto G$.
Here for fixed $(x_j,y_j,z_j)$ the associated (not necessarily unit-normed)
atoms $\widetilde d_{z_j}\in\Ds$ have the form,
\begin{equation}\label{eq_atom_fresnel_1}
  \widetilde d_{z_j}(\cdot-x_j,\cdot-y_j) 
  := \chi_{B_r}(\cdot-x_j,\cdot-y_j) \convbi  \text{Re}(h_{z_j}) \convbi \delta_{x_j,y_j}.
\end{equation}
As before the first step is to calculate the norm of an atom and the correlation of two distinct ones.
Therefore we need some properties of the Fresnel function.
\begin{proposition}
  \label{prop_fresnel_properties}
  For the convolution of the Fresnel function we have the properties~\cite{Liebling2003}
  \begin{align*}
    h_{z_1} \convbi h_{z_2} & = h_{z_1+z_2}, && \mbox{for all $z_1,z_2 \in \R$},\\
    h_{z} \convbi h_{-z} & = \delta, && \mbox{for all $z \in \R$}.
  \intertext{
  With that and $\overline{h_z}=h_{-z}$ we get for the real part of the Fresnel function}
    \text{Re}(h_{z_1}) \convbi \text{Re}(h_{z_2}) & = \tfrac{1}{2} \big(\text{Re}(h_{z_1+z_2}) + \text{Re}(h_{z_1-z_2})\big),
      && \mbox{for all $z_1,z_2 \in \R$},\\
    \text{Re}(h_{z}) \convbi \text{Re}(h_{z}) & = \tfrac{1}{2} \big(\delta + \text{Re}(h_{2z})\big),
      && \mbox{for all $z \in \R$}.
  \end{align*}
\end{proposition}

Another important property is that
the convolution of a function with the Fresnel function---the
so-called Fresnel transform---can be related to a direct multiplication with
its Fourier transform which is defined by:
\[
\Fourier f (\xi,\eta) = \int\limits_{\R^2} f(x,y) \, \exp(-2\pi \iota (x\xi + y\eta)) \rd x \rd y.
\]
\begin{proposition}\label{theorem_fresnel-transform}
  Let $f\in L^2(\R^2)$ and $h_z$ be a Fresnel function. Then
  \[
    \big(f \convbi  h_{z}\big) (\xi,\eta)
    = \Fourier\Big\{\iota\lambda z \, h_z \, f\Big\} 
      \Big(\frac{\xi}{\lambda z},\frac{\eta}{\lambda z}\Big) \, h_z (\xi,\eta).
  \]
\end{proposition}

\begin{proof}
  Let $f\in L^2(\R^2)$, $z\in\R$ and $h_z$ the corresponding Fresnel function. 
  Then rearranging yields to the statement,
  \begin{align*}
    \big( & f \convbi  h_{z}\big) (\xi,\eta)
      = \int\limits_{\R^2} f(x,y) \,
        \tfrac{1}{\iota \lambda z} \exp\big(\tfrac{\iota\pi}{\lambda z} \big((x-\xi)^2+(y-\eta)^2\big)\big) \rd x \rd y\\
    & = \tfrac{1}{\iota \lambda z} \exp\big(\tfrac{\iota\pi}{\lambda z} \big(\xi^2+\eta^2\big)\big)
        \int\limits_{\R^2} f(x,y) \exp\big(\tfrac{\iota\pi}{\lambda z} \big(x^2+y^2\big)\big)
        \exp\big(-2\pi \iota \big(\tfrac{x\xi}{\lambda z}+\tfrac{y\eta}{\lambda z}\big)\big) \rd x \rd y\\
    & = \Fourier\big\{\iota\lambda z \, h_z \, f\big\} 
        \big(\tfrac{\xi}{\lambda z},\tfrac{\eta}{\lambda z}\big) \, h_z (\xi,\eta).
  \end{align*}
\end{proof}

\begin{remark}\label{remark_far-field}
  In praxis $f$ has a bounded and small support w.r.t.~$\sqrt{\lambda z}$.
  With $(x^2+y^2)_{\max}$
  denoting the maximal spatial dimension of $f$
  resp. the maximal spatial extend of the corresponding particle
  the so-called far-field condition
  $\frac{(x^2+y^2)_{\max}}{\lambda z} \ll 1$ holds in the proof of
  proposition~\ref{theorem_fresnel-transform},
  cf.~\cite{Tyler1976}.
  In~\cite{soulez2007holography} e.g., particles of radius at about 50{\textmu}m
  are illuminated with a red laser beam (wavelength  630nm)
  and distance to camera of about 250mm.
  Thus the term $(x^2+y^2)_{\max} /(\lambda z) \approx 3\cdot10^{-4}$ and
  hence $\exp\big(\frac{\iota\pi(x^2+y^2)}{\lambda z}\big)$ is approximatively $1$.
  Under the far-field condition
  we can estimate
  \begin{equation}\label{eq_far_field_cond}
    \big(f \convbi  h_{z}\big) (\xi,\eta)
    \approx \Fourier f
      \Big(\frac{\xi}{\lambda z},\frac{\eta}{\lambda z}\Big) \, h_z (\xi,\eta).
  \end{equation}
  With that for the complex valued diffraction,
  with $\rho^2 := \xi^2+\eta^2$ and $J_\nu$ denoting the first kind Bessel function of order $\nu$ 
  we get
  \[
    \big(\chi_{B_r} \convbi  h_{z}\big) (\rho)
    \approx \frac{r}{\iota \rho} \;
    J_1\Big(\frac{2\pi r}{\lambda z} \rho \Big) \; 
    \exp \Big(\iota \frac{\pi}{\lambda z} \rho^2 \Big),
  \]
  since $\Fourier\chi_{B_r}(\rho)= 2 \pi r^2 \Big[\frac{J_1(2 \pi r \rho)}{2 \pi r \rho}\Big]$ 
  holds (Airy's pattern, vide infra).
  With that for a real valued intensity atom we get
  \[
    \chi_{B_r} \convbi  \text{Re}(h_{z})
    = \text{Re} (\chi_{B_r} \convbi  h_{z})
    \approx \frac{r}{\rho}
    J_1\Big(\frac{2\pi r}{\lambda z}\rho\Big)
    \sin \Big(\frac{\pi}{\lambda z} \rho^2\Big),
  \]
  which corresponds to the model given by Tyler and Thompson in~\cite{Tyler1976}.
\end{remark}
Back to the correlation and---as a special case---the norm of an atom: 
The correlation appears as the autoconvolution, namely
\begin{align*}
    \Ip{\widetilde d_{z_i}(\cdot-x_i,\cdot-y_i)&}{\widetilde d_{z_j}(\cdot-x_j,\cdot-y_j)}\\
    & = \int\limits_{\R^2} \widetilde d_{z_i}(x,y) \; \widetilde d_{z_j}(x-(x_j-x_i),y-(y_j-y_i)) \rd x \rd y \\
    & = \big(\widetilde d_{z_i} \convbi \widetilde d_{z_j}\big)(x_j-x_i,y_j-y_i).
\end{align*}
In the following we assume that all particles are located in a plane parallel to the detector,
i.e. $z:=z_i$ is constant for all $i$.
Then the autoconvolution of an atom appears as
\begin{equation*}
  \widetilde d_{z} \convbi \widetilde d_{z}
    = \chi_{B_r} \convbi \chi_{B_r} \convbi \text{Re}(h_{z}) \convbi  \text{Re}(h_{z}).
\end{equation*}
With proposition~\ref{prop_fresnel_properties} and the formula
\[
C(\rho) = \big(\chi_{B_r} \convbi \chi_{B_r}\big) (\rho) =
\begin{cases}
  2r^2 \cos^{-1} \Big(\frac{\rho}{2r}\Big) - \frac{\rho}{2} \sqrt{4r^2 -\rho^2}
  &\mbox{for }4r^2  > \rho^2,\\
  0, &\mbox{else}.
\end{cases}
\]
we get
\begin{equation}\label{eqn_fresnel_autocorr}
  \widetilde d_{z} \convbi \widetilde d_{z}
    = C \convbi \tfrac{1}{2} \Big[\delta \; + \; \text{Re}(h_{2z})\Big]
    = \tfrac{1}{2} \Big[C \; + \; C \convbi \text{Re}(h_{2z})\Big].
\end{equation}
With remark~\ref{remark_far-field} and since $\Fourier C$ is real valued
we get
\begin{align*}
  C \convbi \text{Re}(h_{2z})
  & = \text{Re} (C \convbi h_{2z})
  \approx \text{Re} \Big(
      \Fourier C 
      (\cdot/\lambda z) \; h_{2z}
    \Big)
  = \Fourier C (\cdot/\lambda z)
    \; \text{Re}(h_{2z})\\
  & = \Fourier\chi_{B_r}(\cdot/\lambda z) \; \Fourier\chi_{B_r}(\cdot/\lambda z)
    \; \text{Re}(h_{2z}).
\end{align*}
In physics it is well known that the Fourier transform of a disc is the Bessel cardinal function,
$\jinc(x) := J_1(x) / x$,
since it is the diffraction of a circular aperture at infinite distance.
Nevertheless, for the sake of completeness and mathematical beauty
we will illustrate this computation:
Since the Fourier transform of a radial function is the 
Hankel transform of order zero (also known as Bessel transform of order zero), cf.~\cite[Theorem~IV.3.3, page~155]{Stein1971},
the Fourier transform of $\chi_{B_r}$ appears, for $\rho^2:=\xi^2+\eta^2$, as
\[
  \Fourier\chi_{B_r} (\rho)
  = 2\pi \int\limits_0^{r} S \, J_0(2\pi \rho S) \rd S
  = \frac{1}{2 \pi \rho^2} \int\limits_0^{2\pi r \rho} S \, J_0(S) \rd S.
\]
In order to solve this definite integral we use $\int S J_0(S) \rd S = S J_1(S)$,
cf.~\cite[equation 5.52\;1.]{Jeffrey2007}, and get
\[
  \Fourier\chi_{B_r} (\rho)
  = 2 \pi r^2 \Big[\frac{J_1(2 \pi r \rho)}{2 \pi r \rho}\Big],
\]
hence the Fourier transform of the circle-circle intersection $C$ appears as
\[
  \Fourier C (\rho)
  = \Fourier\chi_{B_r} (\rho) \, \Fourier\chi_{B_r} (\rho)
  = \tfrac{r^2}{\rho^2} J_1^2(2\pi r \rho).
\]
With that result we can easily calculate the norm of an atom $\widetilde d_z$:
Since $C(0) = \pi r^2$,
$\Fourier C (0) = \big(\Fourier\chi_{B_r} (0)\big)^2 = \big(\int \chi_{B_r} \rd x \big)^2= \pi^2 r^4$
and $h_{2z}(0)=0$ we obtain
\[
  \norm{\widetilde d_z}^2
  = \big|\widetilde d_{z} \convbi \widetilde d_{z}\big| (0)
  \approx \tfrac{1}{2} \pi r^2.
\]
Hence for fixed $z$ we can represent the associated unit-normed
atoms $d_z\in\Ds$, with $R^2:=x^2+y^2$, via
\begin{equation}\label{eq_atom_fresnel_2}
  d_{z}
  := \frac{\widetilde d_{z}}{\norm{\widetilde d_{z}}}
  \approx \Big(\frac{2}{\pi}\Big)^{\frac{1}{2}} \frac{1}{R}
    J_1\Big(\frac{2\pi r}{\lambda z}R\Big)
    \sin \Big(\frac{\pi}{\lambda z} R^2\Big).
\end{equation}
In figure~\ref{bild:fresnel_atom} the centered atom 
for a particle of 50{\textmu}m radius
is displayed which is
illuminated with a red laser beam (wavelength  630nm) in a
distance of 250mm to the camera.

The autoconvolution for general $\rho$
and hence the correlation of two atoms
$d_{z}(\cdot-x_i,\cdot-y_i)$ and $d_{z}(\cdot-x_j,\cdot-y_j)$ 
with distance distance $\rho=((x_j-x_i)^2+(y_j-y_i)^2)^{\frac{1}{2}}$
in digital holography emerges as
\begin{align}
  \Big|\Ip{d_{z}(\cdot-x_i,\cdot-y_i)&}{d_{z}(\cdot-x_j,\cdot-y_j)}\Big|
    = \big|d_{z} \convbi d_{z}\big| (\rho)
    = \frac{1}{\norm{\widetilde d_z}^2} \, \big|\widetilde d_{z} \convbi \widetilde d_{z}\big| (\rho) \notag\\
  & \approx \frac{1}{\pi r^2} \, 
    \Big[
      C(\rho) +
      \Fourier C \Big(\frac{\rho}{\lambda z}\Big) \; \big|\text{Re}(h_{2z}(\rho))\big|
    \Big] \notag\\
  & = \frac{C(\rho)}{\pi r^2} 
    + \frac{1}{4} \,
    J_1^2\Big(\frac{2\pi r}{\lambda z} \rho\Big) \;
    \Big|\sinc \Big(\frac{1}{2\lambda z} \rho^2\Big)\Big|,\label{eq_cor_fresnel}
\end{align}
where $\sinc$ denotes the normalized sine cardinal and is defined via
$\sinc(x) := \sin(\pi x) / \pi x$.

The correlation in digital holography~(\ref{eq_cor_fresnel}) is
not as easily valuable as in mass spectrometry, because
it is not monotonically decreasing in the distance $\rho$
due to the oscillating Bessel and sine functions.
To come to an estimate from above which is monotonically decreasing we
use bounds for the absolute value of the Bessel functions $J_1^2$.
In~\cite{Landau2000} Landau gives estimates for $|J_\nu(x)|$ for $x>0$ and $\nu > 0$, namely
\begin{equation}\label{eq_landau}
  |J_\nu(x)| \leq \min\{b_L \nu^{-1/3}, c_L x^{-1/3}\},
\end{equation}
with constants
\begin{align*}
  b_L & := \sqrt[3]{2} \sup\limits_{x>0} \frac{\sqrt{x}}{3} 
        \Big(J_{-\frac{1}{3}} \Big(\frac{2}{3}x^{\frac{3}{2}}\Big) + J_{\frac{1}{3}} \Big(\frac{2}{3}x^{\frac{3}{2}}\Big)\Big)
        \approx 0.6748,\\
  c_L & := \sup\limits_{x>0} x^{\frac{1}{3}} J_0(x) \approx 0.7857.
\end{align*}
In addition sine cardinal obviously is bounded from above via 1 and $1/x$ and hence we have
\begin{equation}
  \label{eq_estimate_correlation_holography}
  \big|d_{z} \convbi d_{z}\big| (\rho)
  \leq \frac{C(\rho)}{\pi r^2}
  + \frac{1}{4}
  \min \Big\{b_L^2,
  c_L^2 \Big(\frac{\lambda z}{2\pi r}\Big)^{\frac{2}{3}} \rho^{-\frac{2}{3}} \Big\}
  \min \Big\{1,
  \frac{2\lambda z}{\pi} \rho^{-2} \Big\},
\end{equation}
which now is monotonically decreasing in $\rho$.
Figure~\ref{bild:fresnel_cor} illustrates the oscillating part of the correlation~(\ref{eq_cor_fresnel})
and its corresponding upper bound
for two particles of 50{\textmu}m radius
which are illuminated with a red laser beam (wavelength  630nm) in a
distance of 250mm to the camera.

\begin{figure}[ht]
  \begin{minipage}[b]{0.5\linewidth}
    \hspace{-.5cm}
    \begin{tikzpicture}[xscale=0.006,yscale=3]
      \begin{scope}[>=latex]
        \draw[->] (-510,-0.8) -- (510,-0.8) node[above] {$R$(\textmu m)};
        \foreach \y in {-0.4,0,0.4,0.8}
          \draw[help lines] (-510,\y) -- (510,\y);
        \draw[->] (-500,-0.825) -- (-500,0.9) node[above] {$d_z(R)$};
      \foreach \x in {-250,0,250}
          \draw[help lines] (\x,-0.8) -- (\x,0.85);
      \end{scope}
      \foreach \x in {-250,0,250}
        \draw (\x,-0.775) -- (\x,-0.825) node[below] {\footnotesize {\FPeval\result{round(({\x}*10):0)}$\FPprint\result$}};
      \foreach \y in {-0.8,-0.4,0.4,0.8}
        \draw (-480,\y) -- (-520,\y) node[left] {\footnotesize {\FPeval\result{round(({\y}*10):0)}$\FPprint\result\cdot10^{-4}$}};
      \draw (-480,0) -- (-520,0) node[left] {\footnotesize {0}};
      \draw plot file {./fresnel_atom.dat};
    \end{tikzpicture}
    \caption{Unit-normed, centered atom of particles of radius 50{\textmu}m,
      illuminated with a red laser beam (wavelength  630nm) and distance to camera of 250mm.}
    \label{bild:fresnel_atom}
  \end{minipage}
  \begin{minipage}[b]{0.5\linewidth}
    \hspace{0.5cm}
    \begin{tikzpicture}[xscale=0.03,yscale=40]
      \begin{scope}[>=latex]
        \draw[->] (-5,0) -- (160,0) node[above] {$\rho$(\textmu m)};
        \draw[->] (0,-0.005) -- (0,0.12) node[above right] {correlation};
      \end{scope}
      \draw (0,0.005) -- (0,-0.005) node[below] {\footnotesize $0$};
      \foreach \x in {50,100,150}
        \draw (\x,0.003) -- (\x,-0.003) node[below] {\footnotesize {\FPeval\result{round(({\x}*10):0)}$\FPprint\result$}};
      \foreach \y in {0.1,0.05}
        \draw (4,\y) -- (-4,\y) node[left] {\footnotesize {${\y}$}};
      \draw plot file {./fresnel_cor.dat};
      \draw plot file {./fresnel_abschaetz_cor.dat};
    \end{tikzpicture}
    \caption{The oscillating part of the correlation of two atoms
      with distance $\rho$
      and its corresponding monotonically decreasing estimate
      (same settings as in figure~\ref{bild:fresnel_atom}).}
    \label{bild:fresnel_cor}
  \end{minipage}
\end{figure}

\begin{remark}
  In~\cite{Krasikov2006} Krasikov gives more precise estimation for $J_\nu^2(x)$, namely,
  for $\nu>-1/2$, $\varsigma:=(2\nu+1)(2\nu+3)$ and $x>\sqrt{\varsigma + \varsigma^{2/3}}/2$,
  \[
    J_\nu^2(x) \leq \frac{4(4x^2-(2\nu+1)(2\nu+5))}{\pi((4x^2-\varsigma)^{3/2} - \varsigma)}.
  \]
  With that (asymptotically $|d_{z} \convbi d_{z}|(\rho) \sim \rho^{-3}$) instead of Landau's rough bound~(\ref{eq_landau}) 
  (asymptotically $|d_{z} \convbi d_{z}|(\rho) \sim \rho^{-\frac{8}{3}}$) one can
  get a more precise recovery condition for digital holography.
  Since this technical computation is beyond the scope of this theoretical paper
  we postpone it here.
\end{remark}

With this estimation we will come to a resolution bound for droplets jet reconstruction,
as e.g.~used in~\cite{soulez2007holography}.
Here monodisperse droplets, i.e. they have the same size, shape and mass,
were generated and emitted on a strait line parallel to the detector plane.
This configuration eases the computation of the Neumann ERC and the Neumann
ERC in presence of noise.
Analogously to mass spectrometry we define that the particles appear
at some selected points $i\in\Delta\Z:=\set{i\in\Z}{\tfrac{i}{\Delta}\in\Z}$,
where the parameter $\Delta$ describe the dictionary refinement.
If we additionally assume that the particles have the minimal distance
$\rho\in\Delta\N$, then the sum of inner products of support atoms $\Ds(I)$ 
and non-support atoms $\Ds(I^\complement)$
can be estimated from above. For $\rho>\Delta$ we get
\begin{align*}
  \sup\limits_{i\in I} \sum\limits_{\sumstack{j\in I}{j \neq i}} |\ip{d_i}{d_j}|
  & \leq \sum\limits_{j=1}^{\lfloor N/2 \rfloor}
    \tfrac{2}{\pi r^2} C(j\rho)
    + \tfrac{1}{2}
    \min \Big\{
      b_L^2,
      c_L^2 (\tfrac{\lambda z}{2\pi r})^{\frac{2}{3}} \, (j\rho)^{-\frac{2}{3}}
    \Big\}
    \min \Big\{
      1,
      \tfrac{2\lambda z}{\pi} (j\rho)^{-2}
    \Big\}
\intertext{and}
  \sup\limits_{i\in I^\complement} \sum\limits_{j\in I} |\ip{d_i}{d_j}|
  & \leq \sup\limits_{\sumstack{i\in\Delta\Z}{\Delta\leq i \leq \rho-\Delta}} \,
    \sum\limits_{j=-\lfloor N/2 \rfloor}^{\lfloor N/2 \rfloor}
    \tfrac{1}{\pi r^2} C(|j \rho-i|) \\
  & + \tfrac{1}{4}
    \min \Big\{
      b_L^2,
      c_L^2 (\tfrac{\lambda z}{2\pi r})^{\frac{2}{3}} \, |j \rho-i|^{-\frac{2}{3}}
    \Big\}
    \min \Big\{
      1,
      \tfrac{2\lambda z}{\pi} |j \rho-i|^{-2}
    \Big\}.
\end{align*}

\begin{proposition}\label{theorem_fresnel}
  An estimation from above for the ERC (i.e. $\nsr=0$) and
  $\varepsilon$ERC (i.e. $0<\nsr<\frac{1}{2}$)
  for characteristic functions convolved with the real part of the Fresnel kernel is
  for $\rho>\Delta$
  \begin{align*}
    & \sum\limits_{j=1}^{\lfloor N/2 \rfloor}
      \tfrac{2}{\pi r^2} C(j\rho)
      + \tfrac{1}{2}
      \min \Big\{
        b_L^2,
        c_L^2 (\tfrac{\lambda z}{2\pi r})^{\frac{2}{3}} \, (j\rho)^{-\frac{2}{3}}
      \Big\}
      \min \Big\{
        1,
        \tfrac{2\lambda z}{\pi} (j\rho)^{-2}
      \Big\}\\
  & + \sup\limits_{1\leq i < \tfrac{\rho}{\Delta}} \Bigg\{ \,
    \sum\limits_{j=-\lfloor N/2 \rfloor}^{\lfloor N/2 \rfloor}
    \tfrac{1}{\pi r^2} C(|j \rho-i\Delta|)\\
  & \qquad
    + \tfrac{1}{4}
    \min \Big\{
      b_L^2,
      c_L^2 (\tfrac{\lambda z}{2\pi r})^{\frac{2}{3}} \, |j \rho-i\Delta|^{-\frac{2}{3}}
    \Big\}
    \min \Big\{
      1,
      \tfrac{2\lambda z}{\pi} |j \rho-i\Delta|^{-2}
    \Big\}\Bigg\}
    < 1 - 2\,\nsr.
  \end{align*}
\end{proposition}

\begin{remark}
  Same as before for mass spectrometry:
  If the cardinality of the support $N$ is unknown
  one could replace the finite sums by infinite sums.
  These sums exist and can be expressed
  in terms of the Hurwitz zeta function
  $\zeta(\nu,q) := \sum_{j=0}^\infty (q+j)^{-\nu}$, for $\nu>1$, $q>0$,
  and the Riemann zeta function $\zeta(\nu):=\zeta(\nu,1)=\sum_{j=1}^\infty j^{-\nu}$,
  respectively.
\end{remark}

The condition in proposition~\ref{theorem_fresnel} seems not to be easy to handle.
However, in praxis all parameters are known and one can compute a bound
via approaching from large $\rho$.
As soon as the sum is smaller than 1, it is guaranteed that OMP can recover exactly.
A typical setting for digital holography of particles is
the usage of a red laser of wavelength $\lambda=0.6328$\textmu m and
a distance of $z=$ 200mm from the camera, cf.~\cite{soulez2007holography}.
In figure~\ref{bild_epsERC_fresnel} the condition of proposition~\ref{theorem_fresnel}
is plotted for particles with typical radii $r\in\{5, 15, 25, 35, 50, 75\}$\textmu m.
In the computation the asymptotic formula is used, i.e. for an unknown support cardinality $N$.
For the dictionaries a corresponding refinement of $\Delta = r/2$
was chosen.
The colored areas describe the combinations
where the Neumann ERC is fulfilled and hence OMP recovers exactly.

\begin{figure}[ht]
  \centering
  \begin{tikzpicture}[xscale=1.2,yscale=60]
    \filldraw[fill=black!20!white] plot  file {./ERC_fresnel6.dat} |- (4,0);
    \filldraw[fill=black!60!white] plot  file {./ERC_fresnel5.dat} |- (4,0);
    \filldraw[fill=black!20!white] plot  file {./ERC_fresnel4.dat} |- (4,0);
    \filldraw[fill=black!60!white] plot  file {./ERC_fresnel3.dat} |- (4,0);
    \filldraw[fill=black!20!white] plot  file {./ERC_fresnel2.dat} |- (4,0);
    \filldraw[fill=black!60!white] plot  file {./ERC_fresnel1.dat} |- (4,0);
    \foreach \x in {4,6,8,10,12,14}
      \draw (\x,0.002) -- (\x,-0.002) node[below] {\footnotesize {\FPeval\result{round(({\x}*100):0)}$\FPprint\result$}};
    \foreach \y in {0,0.02,0.04,0.06,0.08,0.1}
      \draw (4.1,\y) -- (3.9,\y) node[left] {\footnotesize ${\y}$};
    \begin{scope}[>=latex]
      \draw[->] (4,0) -- (14.5,0) node[right] {$\rho$ (in \textmu m)};
      \draw[->] (4,0) -- (4,0.11) node[above] {$\nsr$};
    \end{scope}
    \draw[->] (5,0.08)    node[above] {$r=75$} -- (6,0.045);
    \draw[->] (6,0.0825)  node[above] {$r=50$} -- (6.5,0.046);
    \draw[->] (7,0.085)   node[above] {$r=35$} -- (7,0.052);
    \draw[->] (8,0.0875)  node[above] {$r=25$} -- (7.5,0.056);
    \draw[->] (9,0.09)    node[above] {$r=15$} -- (8,0.059);
    \draw[->] (10,0.0925) node[above] {$r=5$} -- (8.5,0.055);
  \end{tikzpicture}
  \caption{$\varepsilon$ERC for combinations of $\rho$ and $\nsr$ with
           corresponding dictionary refinement of $\Delta = r/2$.
           For particles the radii $r\in\{5, 15, 25, 35, 50, 75\}$\textmu m
           and the
           asymptotic formula~\eqref{eq_estimate_correlation_holography}
           for an unknown support cardinality $N$ are used.}
  \label{bild_epsERC_fresnel}
\end{figure}
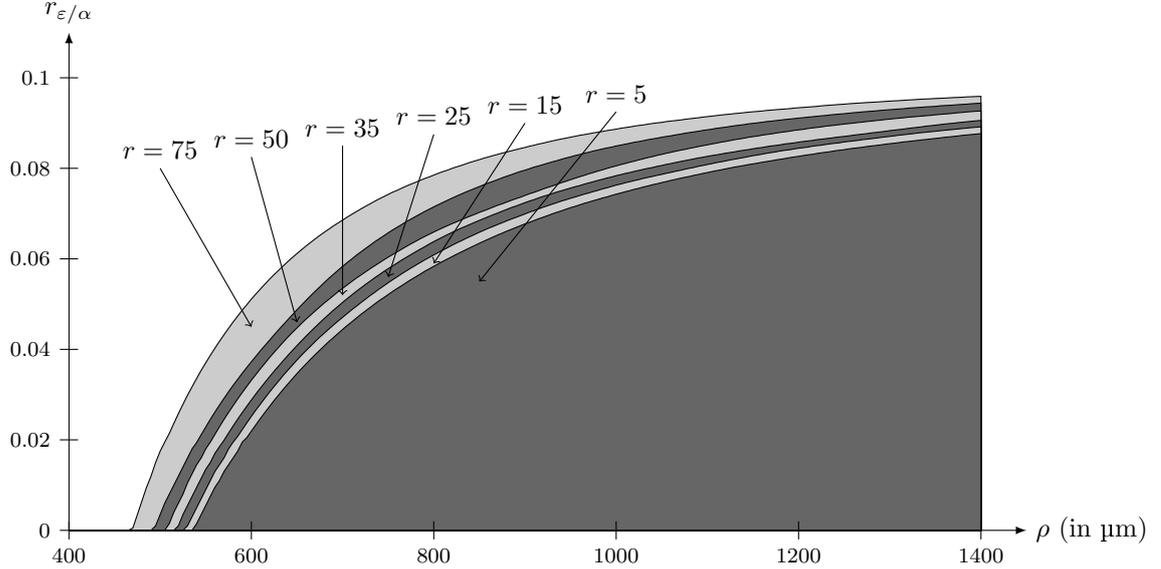

\subsection*{Numerical Examples}
We apply the Neumann $\varepsilon$ERC~(\ref{eq_sufnoise}) to simulated data of
droplets jets. For the simulation we use the same setting as above, i.e. 
a red laser of wavelength $\lambda=0.6328$\textmu m and
a distance of $z=$ 200mm from the camera.
The particles have a diameter of $100$ microns and the corresponding
dictionary the refinement of $25$\textmu m. Those parameters correspond to that of the 
experimental setup used in~\cite{soulez2007holography,soulez2007holography2}

After applying the digital holography model~(\ref{eq_fresnel})
we add Gaussian noise of different noise levels and in each case of zero mean.
For the coefficients we choose $2\alpha_i = 10$ for all $i\in\supp\alpha$.
The figure~\ref{bild_holo} shows three simulated holograms with
different distances $\rho$ and noise-to-signal levels $\nsr$.
For all three noisy examples in the right column all the particles
were recovered exactly. However, only for the image on top($\rho\approx721$\textmu m) the
condition of proposition~\ref{theorem_fresnel} holds.
In the second image in the middle of the figure
the particles have
a too small distance to each other($\rho\approx360$\textmu m) and
even for the noiseless case the condition is not fulfilled.
The last image ($\rho\approx721$\textmu m) was manipulated with unrealistically huge noise
so that here the condition of proposition~\ref{theorem_fresnel}
is violated, too, cf. figure~\ref{bild_epsERC_fresnel}.

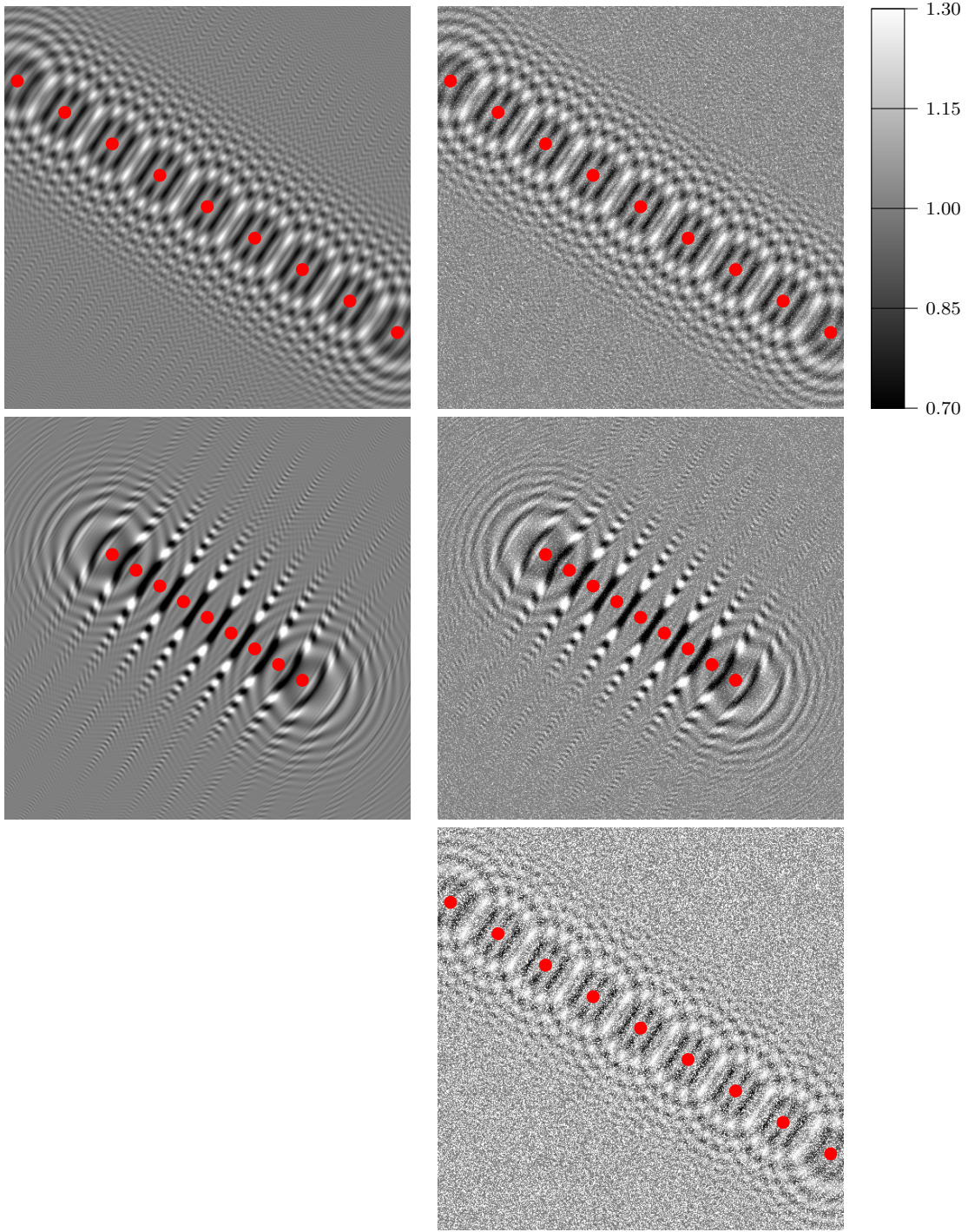
\begin{figure}[p]
  \begin{tabular}{ccc}
    \includegraphics[width=0.47\linewidth]{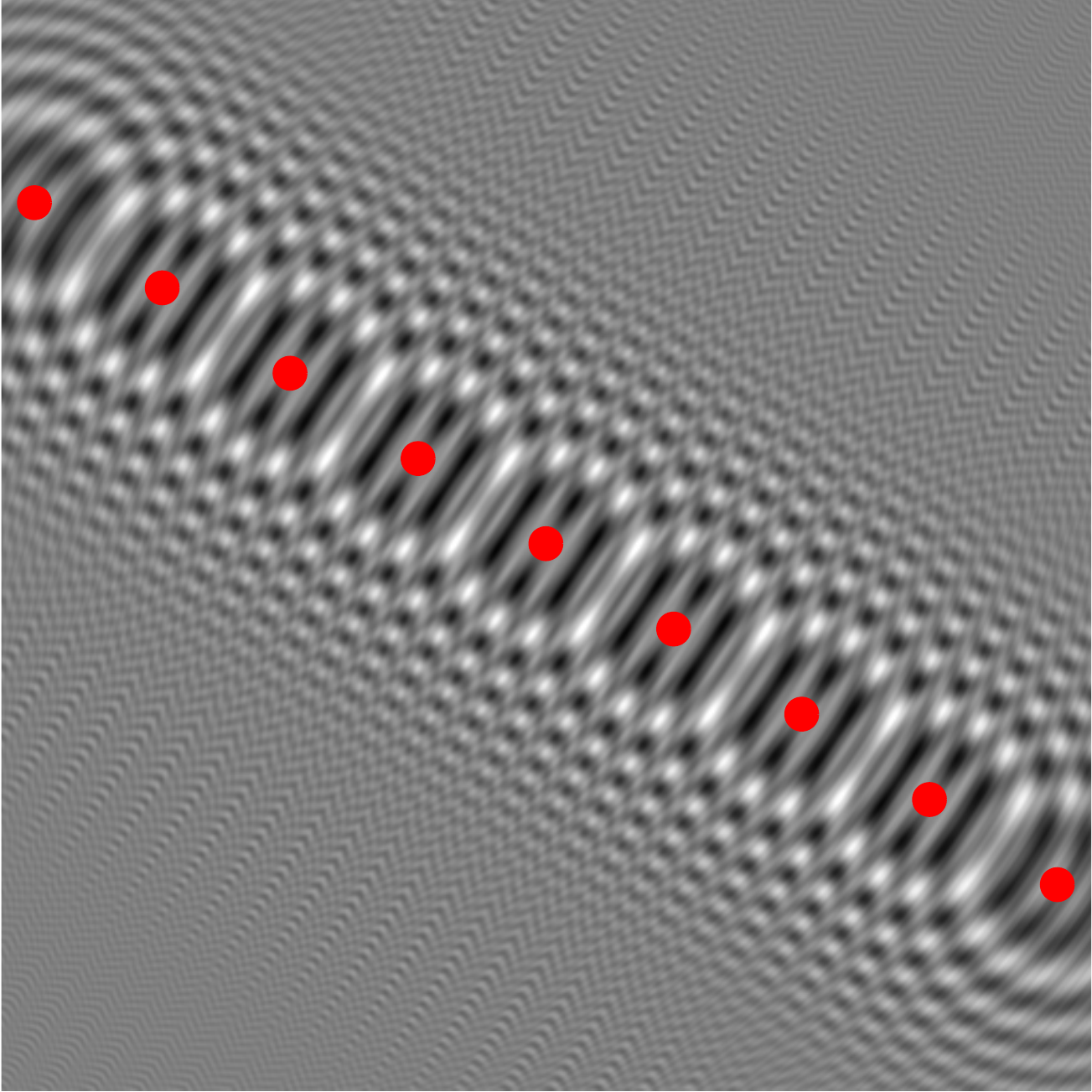}
    & \includegraphics[width=0.47\linewidth]{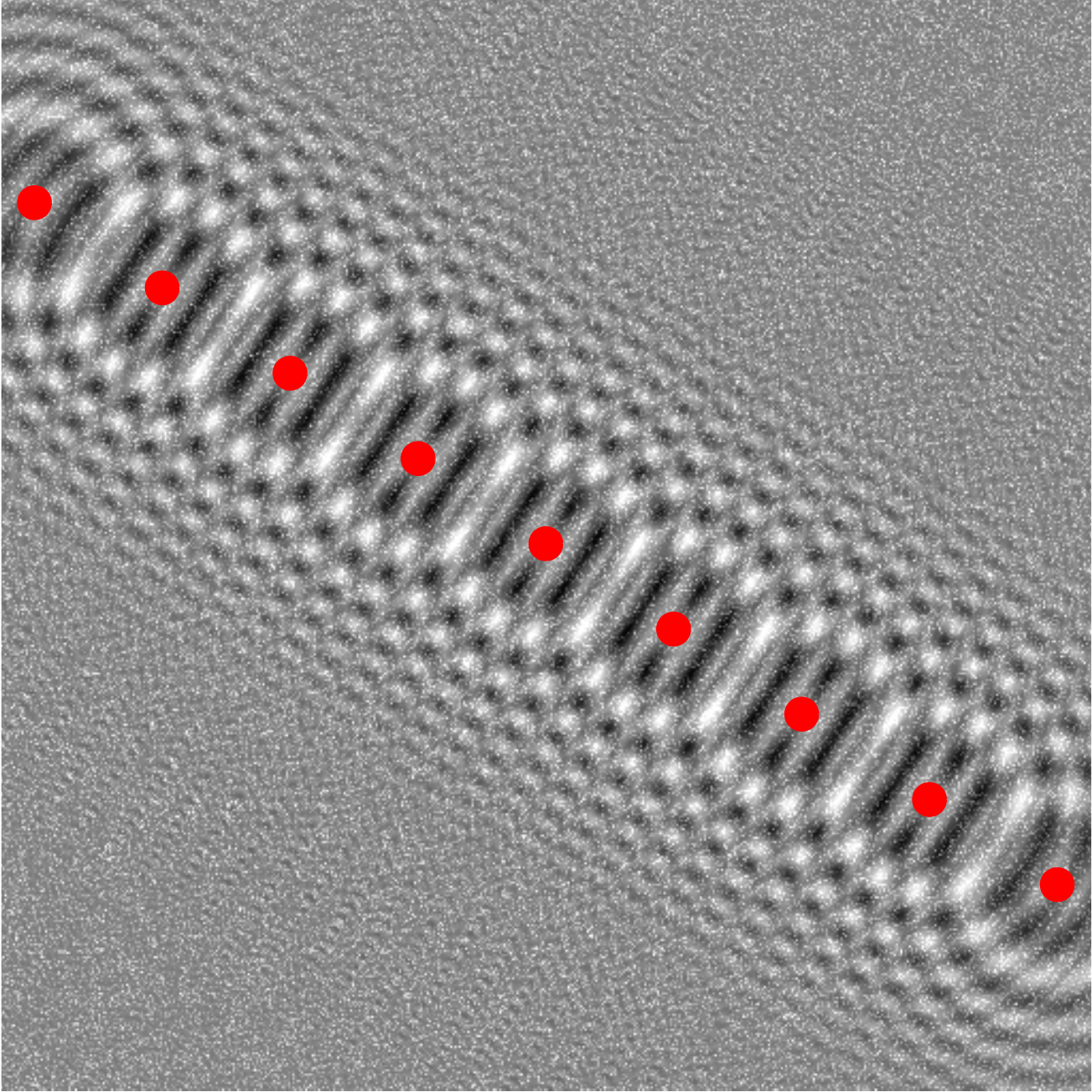}
    &
    \begin{tikzpicture}
      \draw[top color=white,bottom color=black,use as bounding box] (0,0) rectangle (0.5,6.1);
          \foreach \x in {0.0,1.525,3.05,4.575,6.1} 
      \draw (0,\x) -- (0.7,\x) node[right] {\footnotesize {\FPeval\result{round((6/61*{\x} + 0.7):2)}$\FPprint\result$}};
    \end{tikzpicture}\\
    \includegraphics[width=0.47\linewidth]{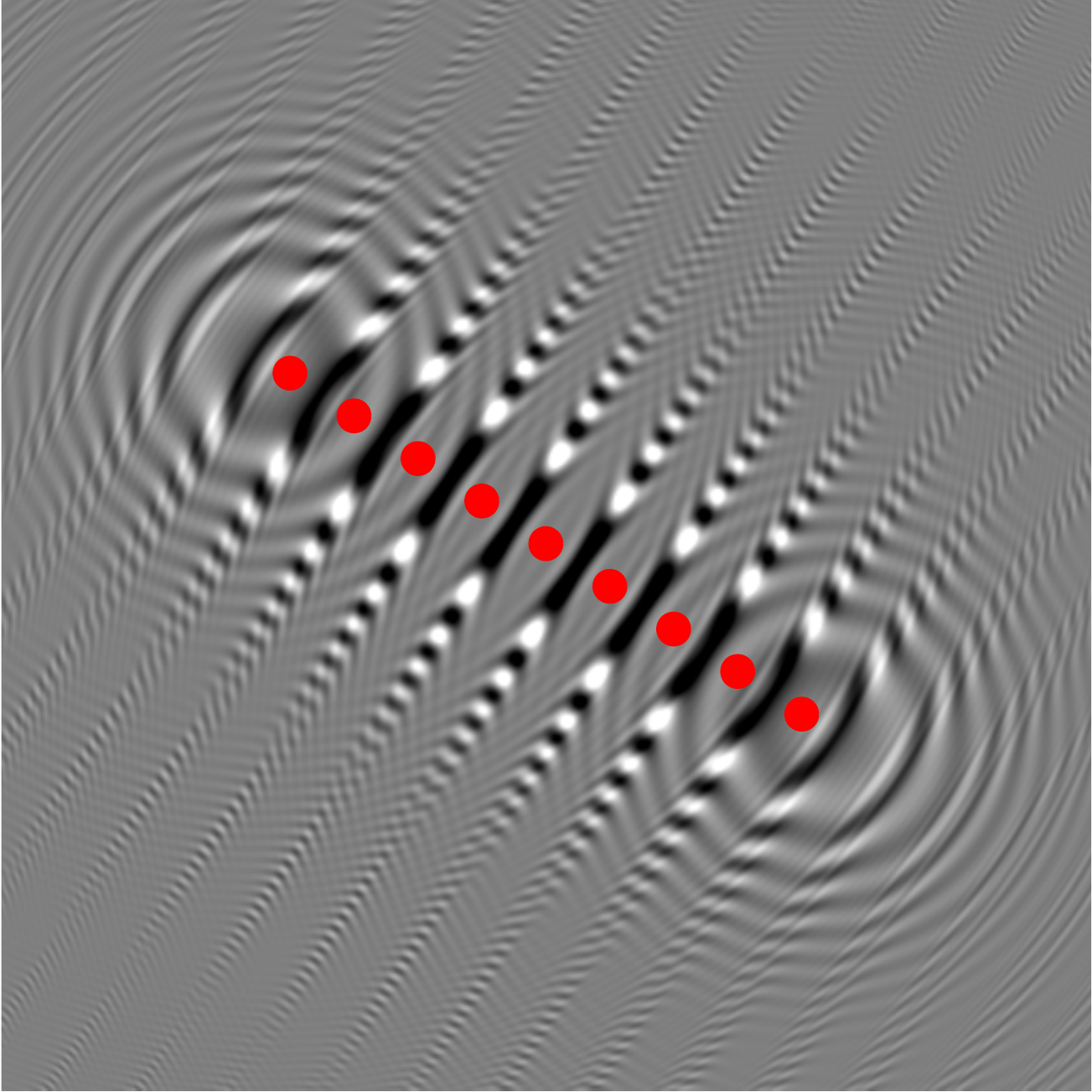}
    & \includegraphics[width=0.47\linewidth]{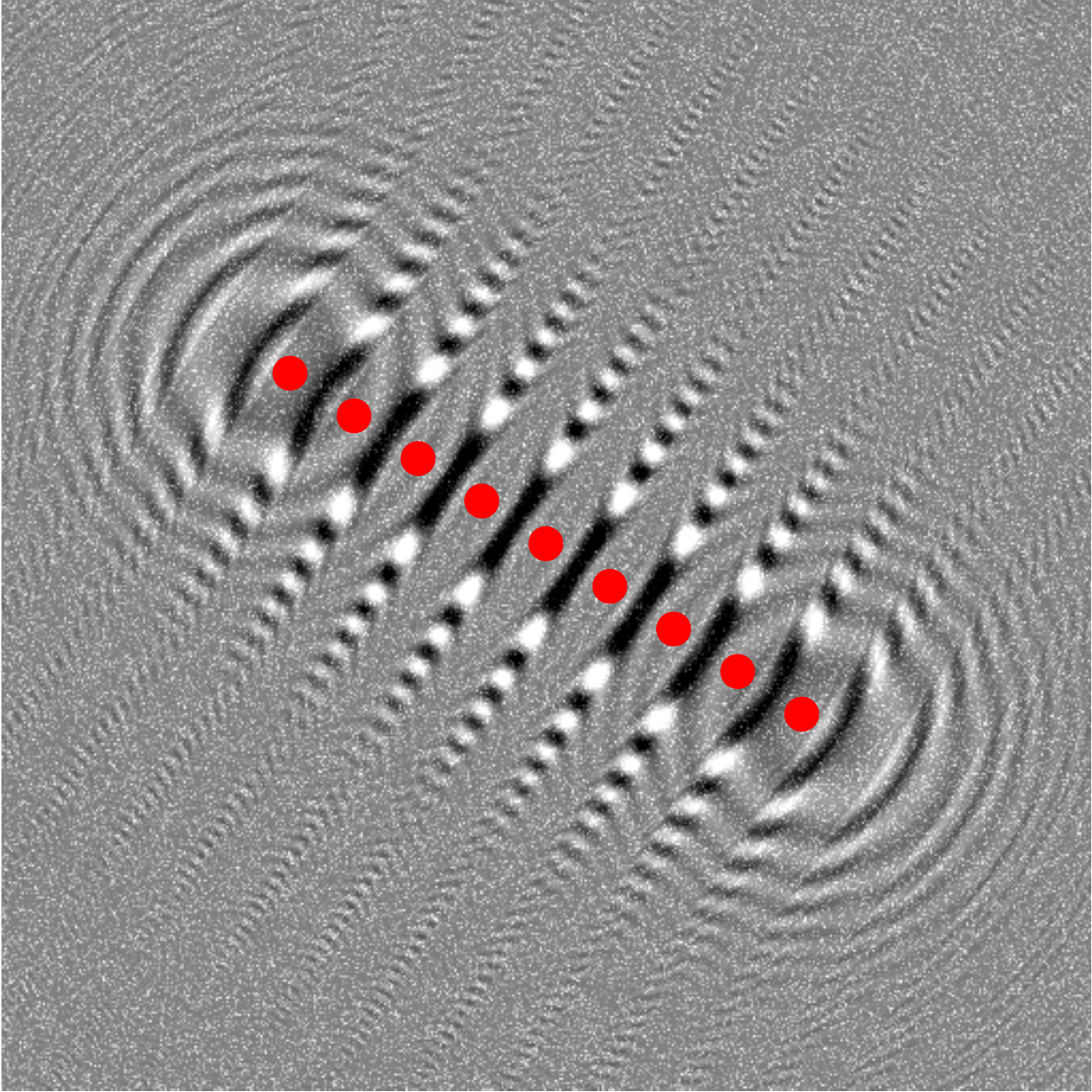}\\
    & \includegraphics[width=0.47\linewidth]{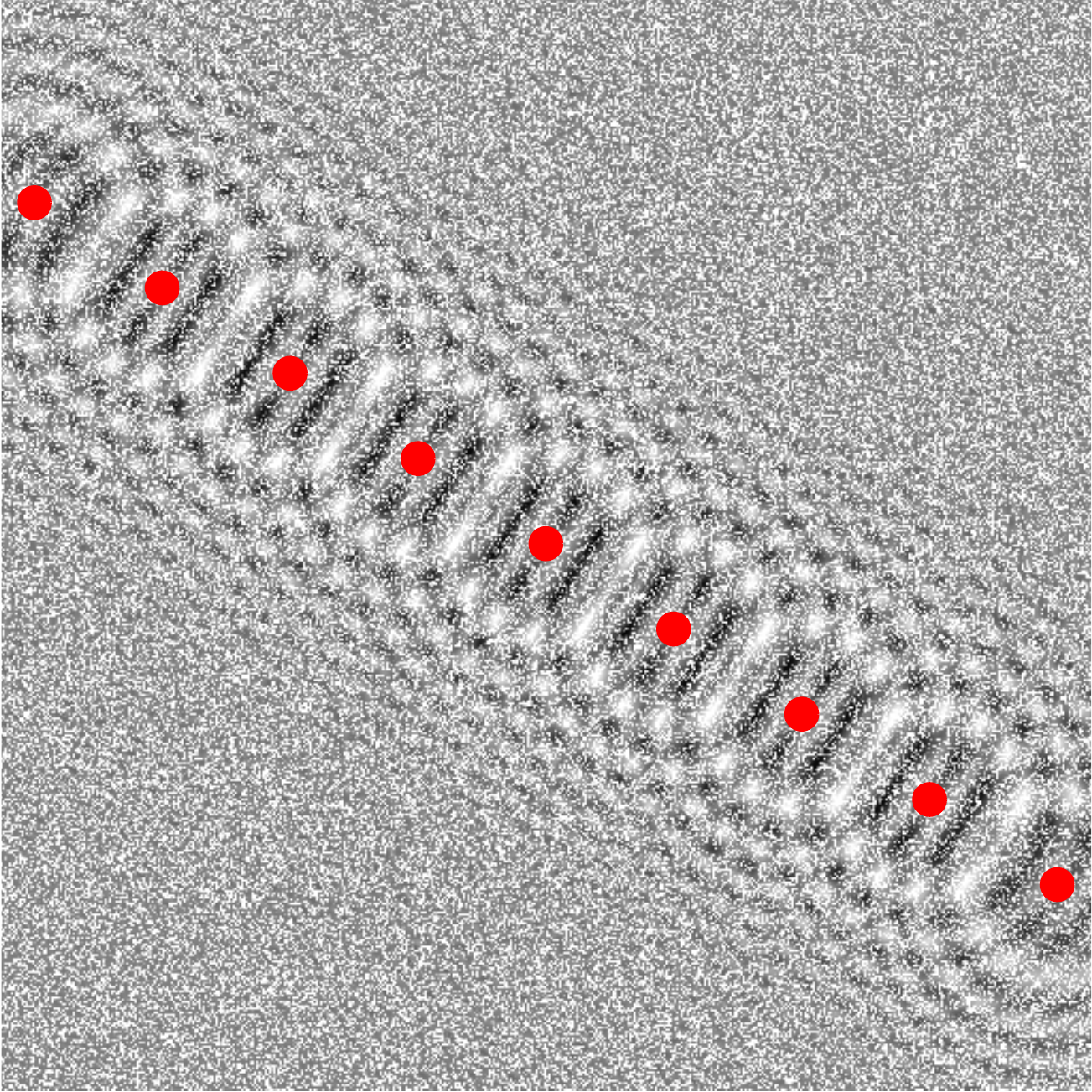}
  \end{tabular}
  \caption{Simulated holograms of spherical particles.
           In the left column the noiseless signals $v$ are displayed.
           For reconstruction the noisy signals $v^\varepsilon$
           of the right column are used.
           The dots correspond to the location  of detected particles with OMP.
           The algortihm recovered all particles exactly,
           however, the condition of proposition~\ref{theorem_fresnel}
           was just fulfilled for the image on top right.
           In the image in the middle the particles have
           a too small distance to each other
           and at the bottom the image was manipulated with unrealistically high noise.}
  \label{bild_holo}
\end{figure}

\section{Conclusion and Future Prospects}
In this paper we gave exact recovery conditions for the orthogonal
matching pursuit for noisy signals that work without the concept of
coherence. Our motivation was to treat ill-posed problems, and in
particular, two problems of convolution type.  We obtained results on
exact recovery of the support for noiseless and even noisy data.
Moreover, for noisy data there is a simple error bound in
proposition~\ref{prop_error_bound_omp} which shows a convergence rate
of $\bigO(\varepsilon)$. The rate of convergence resembles what is known
for sparsity-enforcing regularization with $\ell^p$ penalty term for
$0<p\leq
1$~\cite{grasmair2008sparseregularization,grasmair2008pleq1,Bredies2009b},
moreover, our results also guarantee the exact recovery of the
support.  

In two real-world applications we showed that these condition lead to
computable conditions and hence, are practically relevant. A main tool
here was, that the atoms in the dictionary are shifted copies of the
same shape and that the correlation of the atom depends on the
distance of the atoms only. Once there is a sufficiently decaying
upper bound for the correlation, we are able to apply the Neumann
ERC~\eqref{eq_suf} and the Neumann
$\varepsilon$ERC~\eqref{eq_sufnoise} and obtain computable conditions
for exact recovery as illustrated in the examples in
section~\ref{sec_massspectro} and~\ref{sec_holography}.
However, experiments indicated that the
conditions for exact recovery from theorem~\ref{theorem_suf}
and~\ref{theorem_sufnoise} are too restrictive.  An idea to come to a
tighter exact recovery condition is to bring in more prior knowledge,
as e.g.~a non-negativity constraint, cf.~\cite{Bruckstein2008}.  We
postpone this idea for future work. For the particle digital
holography application even more prior knowledge may be taken into
account, since the objects are not only non-negative but also all
apertures have the same denseness, i.e.~$\alpha_i$ is constant for all
$i\in I$.

As discussed in remark~\ref{rem_continuous_omp}, a straightforward
generalization of our approach to fully continuous dictionaries runs
into problems. Especially it seems that there is little hope to obtain
exact recovery of the support, but maybe one may obtain bounds on how
accurate the support is localized. This is strongly related to the
structure of the dictionary (e.g.~that is consists of shifts of the
same object) and of course related to the correlations.

Finally, a further direction of research may be to investigate other
types of pursuit algorithms like regularized orthogonal matching
pursuit~\cite{Needell2009} or CoSAMP~\cite{Needell2009b}.

\ack{Dirk Lorenz acknowldges support by the DFG under grant LO
  1436/2-1 within the Priority Program SPP 1324 ``Extraction of
  quatitative information in complex systems''. Dennis Trede
  acknowledges support with the BMBF project INVERS under grant
  03MAPAH7.}

\section*{References}

\bibliographystyle{plain}
\bibliography{literatur}
\end{document}